\documentclass[preprint,aos,noautosecdot]{imsart}
\usepackage[square,sort,comma,numbers]{natbib}
\setattribute{journal}{name}{}

\RequirePackage[OT1]{fontenc}
\RequirePackage{amsthm,amsmath,natbib,graphicx,amsfonts}
\RequirePackage[colorlinks,citecolor=blue,urlcolor=blue]{hyperref}
\usepackage{amssymb,tabularx,multicol,multirow,booktabs}
\usepackage{mhequ}
\providecommand{\nopunct}{\spacefactor \@nopunct}
\def\@nopunctsfcode{1007}

\newtheorem{theorem}{Theorem}[section]
\newtheorem{lemma}[theorem]{Lemma}
\newtheorem{remark}[theorem]{Remark}

\newtheorem{cor}[theorem]{Corollary}
\newtheorem{prop}[theorem]{Proposition}

\newtheorem{defn}[theorem]{Definition}

\newtheorem{example}[theorem]{Example}
\usepackage{color}
\usepackage{amssymb}
\usepackage{comment}
\usepackage{pdfpages}
\usepackage{graphicx}
\usepackage{dcolumn}
\usepackage{subcaption}

\usepackage{tikz}
\usetikzlibrary{shapes,backgrounds}
\usepackage{dcolumn}
\usetikzlibrary{arrows,positioning}

\RequirePackage[numbers]{natbib}

\numberwithin{equation}{section}

\def \be{\begin{equs}}
\def \ee{\end{equs}}
\def \P{\mathbb{P}}
\def \E{\mathbb{E}}

\def \indT{\mathrm{1}_\T}
\def \ind{\mathrm{1}}
\def \tnaive{\widehat{t}_{\mathrm{Neyman}}}
\def \idt{t_\mathrm{ideal}}

\def \cp{C_{{\mathcal{P}}}}
\def \T{\mathrm{T}}

\def \w{\d_{\mathrm{w}}}
\def \d{\mathbf{d}}
\def \dk{\d_{K}}

\def \dkp{\d}
\def \wdkp{\mathcal{W}, \dkp}
\def \TV{\mathrm{TV}}

\def \pow{\mathrm{pow}}

\DeclareMathOperator\Var{Var}
\DeclareMathOperator\Cov{Cov}
\DeclareMathOperator\Corr{Corr}
\newcommand\pq[2]{\chi^{#1}_{#2}}
\newcommand\dmax{d_{\max}}
\newcommand\dmin{d_{\min}}

\begin{document}

\begin{frontmatter}
\title{Designs for estimating the treatment effect in Networks with
Interference}

\begin{aug}
\author{\fnms{Ravi} \snm{Jagadeesan}\thanksref{t1,m1}\ead[label=e1]{rjagadeesan@college.harvard.edu}}
\and
\author{\fnms{Natesh} \snm{S. Pillai}\thanksref{t2,m1}\ead[label=e2]{pillai@fas.harvard.edu}}
\and
\author{\fnms{Alexander} \snm{Volfovsky}\thanksref{t3,m2}\ead[label=e3]{alexander.volfovsky@duke.edu}}
\thankstext{t1}{Undergraduate Student, Harvard University}
\thankstext{t2}{Associate Professor, Department of Statistics, Harvard University}
\thankstext{t3}{Assistant Professor, Department of Statistical Science, Duke University}
\runauthor{Jagadeesan, Pillai and Volfovsky}
\affiliation{
Harvard University\thanksmark{m1} and
Duke University \thanksmark{m2}
}
\end{aug}
\begin{abstract}
In this paper we introduce new, easily implementable designs for drawing causal inference from randomized
experiments on networks with interference. Inspired by the idea of
matching in observational studies, we introduce the notion of considering a 
treatment assignment as a ``quasi-coloring" on a graph. Our idea of a perfect quasi-coloring strives to match every treated unit on a given network with a distinct control unit that has 
identical number of treated and control neighbors. 
For a wide range of interference functions encountered in applications, we show both by theory and simulations that the classical Neymanian estimator for the direct effect has desirable properties for our designs. This further extends to settings where homophily is present in addition to interference.
\end{abstract}
\end{frontmatter}
\section{Introduction.}
In this paper, we construct and analyze new designs for estimating treatment effects from randomized experiments in networks with interference.
With the proliferation of network data and the steady increase in the number of experiments conducted on networks, understanding the behavior of individuals in a network has become an important issue in many scientific fields.
Epidemiologists study the transmission of disease over social networks \cite{aiello2016design}, computer scientists are interested in information diffusion in large computer networks \citep{yang2010predicting,gruhl2004information} and sociologists study the effects of school integration on friendship networks \citep{moody2001race}. 
While much of the early statistical work on networks focussed on models for understanding network formation \citep{hoff2002latent,holland1983stochastic},  there has been a recent surge in drawing causal inference from experiments on networks \citep{shalizi2011homophily,eckles2017design,sussman2017elements,toulis2017elements,toulis2013estimation}. 


A time-honored approach to performing causal inference from randomized experiments entails the following steps \citep{holland1986statistics,rubin1978bayesian,rubin1974estimating}: (i) define the population of units, (ii) define the treatment assignment and (iii) define the quantity, or estimand, 
of interest. When an experiment is conducted on a network, we must revisit each of these elements. First, the object of inference can be the network, the edges of the network or the nodes of the network. 
We focus on the case where the nodes are the experimental units and our population is just the observed units.
Next, the treatment assignment mechanisms proposed in this paper are conditional on a given network and thus the events that any two units receive treatment are not independent. This is in stark contrast to usual Bernoulli-type randomization mechanisms where treatment is assigned to units independently or with very weak dependence. Finally, our estimand  of interest is the \emph{direct} treatment effect (effect of treatment on the treated unit irrespective of the treatment status of the rest of the network) discussed below.


Much of the current works on causal inference on networks study generic Bernoulli-type randomization schemes and construct various estimators for minimizing their Mean-Squared Error (MSE); a notable exception is the recent work \citep{eckles2017design}. In contrast, we fix an estimator of interest and focus on the \emph{design} of treatment assignments. 
We study the classical Neymanian estimator that takes the difference between the means of the outcome for treated nodes and the control nodes. Our approach is motivated by two key reasons: (i) The Neymanian estimator is ubiquitously used. It is a natural estimator for the direct effect and improves on reweighted versions of it (such as Horvitz-Thompson, Hajek, \textit{etc}.) due to its \emph{prima facie} interpretability.
(ii) It has been emphasized by many researchers that for objective causal inference, ``design trumps analysis" \cite{rubin2008objective}. It is known that this estimator is biased under standard designs such as Bernoulli trials (every unit has probability of treatment $p$) and a completely randomized design (a fraction $p$ of the units is assigned to treatment). We consider a more natural randomization scheme that works to remove the effects of interference and homophily by balancing the relevant distributions between treated and untreated nodes. 

\begin{figure} 
\begin{tikzpicture}[scale=0.5]
\draw (0 , 0) --(4, 0) -- (4,4) -- (0,4)--(0,0);
\draw  (8,0)--(12,0)--(12,4)--(8,4)-- (8,0);
\draw [fill] (0,0) circle [radius=.15];
\draw [fill] (4,0) circle [radius=.15];
\draw [fill,gray](0,4) circle [radius=.15];
\draw [fill,gray](4,4) circle [radius=.15];
\draw [fill] (8,0) circle [radius=.15];
\draw [fill] (12,4) circle [radius=.15];
\draw [fill,gray](8,4) circle [radius=.15];
\draw [fill,gray](12,0) circle [radius=.15]; 
\end{tikzpicture}
\caption{The coloring of $G$ on the left is
a perfect quasi-coloring, since both the black vertices and the gray vertices have exactly two neighbors of opposite colors.
The coloring of $G$ on the right is not a perfect quasi-coloring because both the black vertices have two gray neighbors
where as both the gray vertices have two black neighbors.} 
\label{fig:quasicol}
\end{figure}
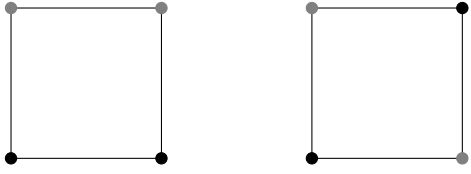 
Conceptually, our main contribution is the idea of considering a treatment assignment as a ``quasi-coloring'' of a graph (see Definition \ref{defn:perfect}). 
Roughly speaking, a treatment assignment is a \emph{perfect quasi-coloring}\footnote{The word coloring is reserved for something specific in graph theory; thus we use the phrase ``quasi-coloring".}, if
for every treated vertex $v$ (represented by black dots, say), there is a non-treated vertex $v'$ (represented by gray dots) that has the same number of treated and non-treated neighbors as that of $v$. Thus having a perfect quasi-coloring on a graph $G$ ensures that one can color the graph in such a way that for every black vertex, there exists a \emph{distinct} gray vertex with identically colored neighbors. Figure \ref{fig:quasicol} shows two instance of coloring a square, where one is a perfect quasi-coloring and the other is not.
For multiple treatments, this definition can be naturally extended to perfect quasi-colorings with $k$ colors.\par
Our notion of perfect quasi-colorings is inspired by the idea of covariate balance in the context of matching in observational studies.
For any given network,  if a treatment assignment mechanism satisfies our notion of quasi-coloring, we prove that the Neymanian estimator for the direct treatment effect is unbiased for a wide range of families of interference effects encountered in practice. This replicates the behavior of the Neymanian estimator in classical randomized experiments. It turns out that, for many graphs, perfect quasi-colorings are not available or may be very difficult to construct. To circumvent this issue, we develop treatment assignment mechanisms that correspond to ``approximately perfect quasi-colorings''. The closer an approximately perfect quasi-coloring is to a perfect quasi coloring, the smaller its bias. Based on this notion we develop a new \emph{restricted randomization} design that reduces bias and variance. In networks where a perfect quasi-coloring is not possible, we give easily implementable algorithms to construct designs with desirable properties; see the ``partitioning  by degree" design in Definition \ref{def:pbd}.

We give bounds for the bias and variance of our estimator under a few different settings of approximate quasi-colorings. These results are then used to prove asymptotic consistency of our estimator in both dense and sparse asymptotic regimes for network growth. We also derive bounds for the MSE of the Neymanian estimator under homophily. We demonstrate the efficacy of our proposed randomization scheme in a series of simulations --- varying both the type of interference and the network. Our proofs for the dense \textit{vs.} sparse graphs 
are different and thus are of independent interest.
\subsection{Background and literature.}
We briefly survey the relevant literature, point out connections to the present paper and place our work in a broader context. In situations when the experimental units are connected in a network, some of the usual assumptions used in other settings are not likely to hold. For example, the stable unit treatment value assumption \citep{rubin1974estimating} requires that the outcome for a unit only depends on its own treatment, and in particular is independent of the treatment assignment mechanism.
For networks this can be violated in several ways: it is likely that either the behavior of connected units is similar (\emph{homophily}), that their outcomes are associated with the treatment of their network neighbors (\emph{interference}) or that the treatment effect passes temporally across the network (\emph{contagion}). 
It has been previously demonstrated that while these can affect causal inference on a network differently, they are difficult, if not impossible, to distinguish \citep{shalizi2011homophily}. 
These complications lead to a difficulty in specifying an estimand of interest \citep{hudgens2008toward}. The four main estimands in the presence of a network are (i) the effect of treatment were it applied to the whole network versus no one in the network (\emph{total network treatment effect}), (ii) the direct effect of treatment on the treated unit irrespective of the treatment status of the rest of the network (\emph{direct treatment effect}), (iii) the spillover effect of treatment of the network on a single unit irrespective of its treatment (\emph{indirect treatment effect}), and (iv) the sum of the direct and indirect effect (\emph{total nodal treatment effect}). 

Different estimands lead to different inference procedures -- both from a design and an analysis point of view. We focus on the design of experiments targeting the direct treatment effect. Other recent work has targeted different estimands: In \citet{choi2016estimation} the author studies estimators for monotone treatment effects and constructs asymptotically consistent bounds for such estimates. 
\citet{eckles2017design} study total network effects by considering a cluster-randomized-design in conjunction with Horvitz-Thompson and Hajek estimators. 
In \citet{sussman2017elements}, the authors construct unbiased estimators for direct and indirect treatment effects for a fixed design. 

The above works study the effects of interference on estimation and make the common assumption that the interference is limited to the immediate neighborhood of a node. We will also make this assumption but our work can be easily generalized to different patterns of interference; see Discussion for more on this point. Another simplifying assumption that is frequently made requires the interference effect to be symmetric -- that is each interfering unit contributes the same indirect effect. We demonstrate results under several classes of interference patterns that generalize this assumption.

\subsection{Notation.}
Fix $n \in \mathbb{N}$ and let $G$ be a graph with $|V(G)| = rn$.  Throughout the paper, we will assume that $G$ has no isolated vertices.
Let $\mathcal{N}(v)$ denote the set of neighbors of a vertex $v \in V(G)$ and let $d(v) = |\mathcal{N}(v)|$ denote the degree of $v$. Also define the minimum and maximum degrees:
\be \label{eqn:dmin}
\dmin = \min_{v \in V(G)} d(v), \quad d_{\max} = \max_{v \in V(G)} d(v).
\ee
We will denote by $\binom{V(G)}{k}$, the set containing all subsets of $V(G)$ with cardinality $k$. 
Similarly, for $1 \leq m_k \leq rn$ with $\sum m_k = rn$, define
\be
\binom{V(G)}{m_1,\ldots,m_k} = \Big \{(A_1,\ldots,A_k), \,A_k \subset V(G), \,|A_k| = m_k,\, A_k \cap A_\ell = \emptyset, \forall k \neq \ell \Big \}.
\ee 
In particular, $\binom{V(G)}{r,\ldots,r}$ denotes the set of all partitions of $V(G)$ into sets of size $r$. For $ r \in \mathbb{N}$, the set $\{1,2,\dots, r\}$ will be denoted by $[r]$. For sets $A, B \subset V(G)$, $A \Delta B$ denotes the set difference.\par

For $\T \subset V(G)$, let $\indT(\cdot)$ denote the indicator function: 
\be
 \indT(v) = \begin{cases}
1 & \text{if } v \in \T, \\
0 & \text{if } v \notin T.
\end{cases}
\ee
For $p \in \mathbb{N}$, we will have $pn$ treated units, and thus $|\T| = pn$.
Set $q = r-p$.
For $\T \subseteq V(G)$ and $v \in V(G),$ let
\[\pq{\T}{v} = \begin{cases}
q & \text{if } v \in \T, \\
-p & \text{if } v \notin \T.
\end{cases}\]


\subsection{Paper guide.}
In Section \ref{sec:model} we introduce a basic model for interference and the Neymanian estimator. In Section \ref{sec:Resran} we discuss some restricted randomizations.
Section \ref{sec:SIFmodel} contains a symmetric interference model.
In Section \ref{sec:PQC} we define our notion of quasi-coloring. We derive the bounds for the MSE of the Neymanian estimator in Section \ref{sec:NeyMSE}. Section \ref{sec:inttypes} introduces a generalization of the symmetric interference model from Section \ref{sec:SIFmodel}. In Section \ref{sec:homophily} we study the effects of homophily on the treatment effect. The results from a simulation study are given in Section \ref{sec:sims}. We close with a short discussion. The proofs for various technical results are given in the appendices.

\section{The Model and the Estimator.}\label{sec:model}
For each vertex $v \in V(G)$, let $x_v,t_v \in \mathbb{R}$ be constants and let $f_v: 2^{\mathcal{N}(v)} \mapsto \mathbb{R}$ be a function such that $f_v(\emptyset) = 0$ for all $v \in V(G)$. We study the  linear model:
\be \label{eqn:linmod}
y_v = x_v + \indT(v)\, t_v + f_v(\T \cap \mathcal{N}(v)), \quad v \in V(G)
\ee
where $\T \subset V(G)$ denotes the treatment group.
In general, the quantity $x_v$ can be thought of as vertex specific attributes, such as covariates. When no covariates are observed, it simply reflects the outcome for node $v$ under control. The function $f_v$ denotes the interference effect. For every vertex $v$, it is only a function of 
its treated neighbors $\T \cap \mathcal{N}(v)$. 

This model (without observed covariates) is a member of the class of neighborhood interference models introduced by \citep{sussman2017elements}. In particular, they demonstrate that this parametrization is equivalent to the potential outcomes notation of \cite{rubin1974estimating} under specific assumptions on the additivity and symmetry of the effects. In particular, Equation~\ref{eqn:linmod} corresponds to the additivity of main effects assumption (ANIA) --- the second most general model in that paper. That is, $x_v$ is the baseline, $t_v$ is the direct treatment effect (defined as the effect of treatment on node $v$ when no one else is treated) and $f_v$ is the interference effect. While \citet{sussman2017elements} construct new estimators for the average treatment effect, we focus on better designs for the Neymanian estimator defined below.


Define the average direct treatment effect as:
\be \label{eqn:avgte}
\bar{t} = \frac{1}{rn} \sum_{v \in V(G)} t_v.
\ee
We are interested in estimating $\bar{t}$.
Throughout the paper, we will have $r$ groups of experimental units, and in each group $p$ units will receive treatment.
When $|\T| = pn$, define the Neymanian estimator
\be \label{eqn:tnaive}
\tnaive = \frac{1}{pqn} \Big(q\sum_{v \in \T}  y_v - p\sum_{v \in V(G)\setminus \T} y_v\Big) = \frac{1}{pqn} \sum_{v \in V(G)} \pq{\T}{v} y_v.
\ee
When $p = q =1$ and $r = 2$, the estimator $\tnaive$ has the usual form:
\be 
\tnaive = \frac{1}{2n} \Big(\sum_{v \in \T}  y_v - \sum_{v \in V(G)\setminus \T} y_v\Big).
\ee
Define the quantity
\be \label{eqn:ideal}
\idt = \frac{1}{pqn} \Big(q\sum_{v \in \T}  (x_v + t_v) - p\sum_{v \in V(G)\setminus \T} x_v \Big).
\ee
The difference\footnote{The quantity $\xi$ is a function of the treatment $\T$, but we suppress this dependence for notational convenience.}
\be \label{eqn:xi}
\xi = \tnaive - \idt = \frac{1}{pqn} \sum_{v \in V(G)} \pq{\T}{v} f_v(\T \cap \mathcal{N}(v)).
\ee
is the ``average interference effect".
Next, we show that bounds on $|\mathbb{E}_\T(\xi)|$ lead to bounds on the bias of $\tnaive$. Here, $\E_\T$ denotes that the expectation is taken over the treatment assignment mechanism.
\begin{lemma}
\label{lem:idtUnbiased}
Suppose that $\T \subset V(G)$ is selected in a fashion so that $\P(v \in \T) = \frac{p}{r}$ for all $v \in V(G)$.  Then, $\mathbb{E}_{\T}(\tnaive) - \bar{t} = \mathbb{E}_{\T}(\xi)$. \end{lemma}
\begin{proof}
The quantity $\idt$ in \eqref{eqn:ideal} can be written as
\be
\idt = \frac{1}{pqn} \sum_{v \in V(G)}\pq{\T}{v} (x_v + \indT(v)\,t_v).
\ee
Since $\P(v \in \T) = \frac{p}{q}$ for all $v \in V(G),$ we obtain that
\begin{align*}
\E_\T(\idt) &= \frac{1}{pn} \sum_{v \in V(G)} \P(v \in T) t_v - \frac{1}{pqn} \sum_{v \in V(G)} x_v \E_\T(\pq{\T}{v})
= \bar{t}.
\end{align*}
Thus
\be
\E_\T(\tnaive) - \bar{t} = \E_\T(\idt + \xi) - \bar{t} = \E_\T(\xi)
\ee
proving the lemma.
\end{proof}

\section{Restricted Randomizations.}\label{sec:Resran}
Fix a partition $\mathcal{P} = (S_1,\ldots,S_n) \in \binom{V}{r,\ldots,r}$ of the vertices $V$ into sets $S_i = \{w^1_i,\ldots,w^r_i\}$.
Define a random vector $\vec{B} = (B_1,\ldots,B_n) \in \binom{[r]}{p}^{n}$ with $B_i$ i.i.d. uniform on $\binom{[r]}{p}$. Conditional on the partition $\mathcal{P}$, we define our treatment assignment mechanism to be:
\be \label{eqn:tbp}
\T_{\vec{B},\mathcal{P}} = \{w^j_i \mid j \in B_i\}.
\ee
Thus we will give treatment to the vertex $w^j_i$ when $j \in B_i$.

The usual Completely Randomized Design (CRD) for treatment assignments is recovered when $\mathcal{P}$ is sampled uniformly from the set $\binom{V(G)}{r,\ldots,r}$. In this section, we obtain bounds for the bias of $\tnaive$ with treatment group $\T_{\vec{B},\mathcal{P}}$ for a fixed partition $\mathcal{P} \in \binom{V(G)}{r,\ldots,r}$.

\subsection{General upper bound on bias.}
The following definition introduces a useful framework for quantifying the variability of the interference effect $f_v$ across the units.
\begin{defn}
The function $f_v$, $v \in V(G)$, is called $K_v$-Lipshitz if
\be \label{eqn:lip}
\left|f_v(A) - f_v(B)\right| \le \frac{K_v|A \Delta B|}{d(v)}
\ee
for $K_v > 0$ and all $A,B \subset \mathcal{N}(v)$.
\end{defn}
Thus the Lipshitz constant $K_v$ provides an upper bound on the amount that treating a proportion of the neighbors of $v$ can affect $y_v$.

\begin{example}
The linear interference function $f_v(A) = \gamma |A|$ for $\gamma \in (0,1)$ is $\gamma d(v)$-Lipschitz.  Moreover, the normalized linear interference function $f_v(A) = \gamma \frac{|A|}{d(v)}$ is $\gamma$-Lipschitz.
\end{example}

The following lemma bounds the bias of $\tnaive$ with treatment $\T_{\vec{B},\mathcal{P}}$ when $f_v$ is Lipshitz.
The idea of the proof is to apply Lemma~\ref{lem:idtUnbiased} to reduce to bounding the expectation of $\xi$.
The Lipshitz condition yields a termwise bound on $\xi$ in (\ref{eqn:xi}).
Given $v \in V,$ let 
\be \label{eqn:pv}
\mathcal{P}_v = S_i,  \quad v \in S_i.
\ee

\begin{lemma}
\label{lem:biasGen}
Suppose the function $f_v$ is $K_v$-Lipshitz.
Then for the partition $\mathcal{P}$ and the treatment assignment in \eqref{eqn:tbp}, we have
\be \label{lem:indunbiased}
\left|\mathbb{E}_\T(\xi)\right| \le \frac{1}{nr(r-1)} \sum_{v \in V(G)} \frac{\left|\mathcal{P}_v\cap \mathcal{N}(v)\right|}{d(v)} K_v.
\ee
\end{lemma}

Lemma \ref{lem:indunbiased} has the following important corollary:
\begin{cor} \label{cor:keycor}
If every element of $\mathcal{P}$ is an independent set in $G$, \textit{i.e.}, if $\{v,v'\} \notin E(G)$ whenever $ \{v,v'\} \subseteq S_i$, and $\T = \T_{\vec{B},\mathcal{P}},$ then $\E_\T(\tnaive) = 0.$
\end{cor}
\begin{proof}
Indeed, in this case, the right hand side of Lemma~\ref{lem:biasGen} has no terms in this case. Thus we have $|\mathbb{E}_\T(\xi)| = 0$ and the proof follows from 
Lemma~\ref{lem:idtUnbiased}.
\end{proof}
Thus Corollary \ref{cor:keycor} implies that if we choose clusters $(S_i)$ of independent sets and then randomize within those clusters for the treatment assignment, then $\tnaive$ will be unbiased. Thus a design principle will be to ensure that elements of $\mathcal{P}$ do not contain too many edges of $G$ using appropriate randomizations.

\subsection{Random choices of \texorpdfstring{$\mathcal{P}$}{P}.}
In this section, we assume that the function $f_v$ is $K_v$-Lipshitz.
Define the \emph{average Lipschitz constant}
\be \label{eqn:kbar}
\bar{K} = \frac{1}{rn} \sum_{v \in V(G)} K_v.
\ee
\begin{example} \label{ex:lip1}
Let $f_v(A) = \gamma |A|$ for some $\gamma  \in (0,1)$. Then $f_v$ is $K_v$-Lipshitz with $K_v = \gamma d(v)$.  When the underlying graph $G$ has average degree $m$, $\bar{K} = \gamma m$.
\end{example}

Choosing $\mathcal{P}$ randomly can help reduce the bias, as the following proposition shows.
As will be seen in the sequel, it will be helpful to restrict the randomization of $\mathcal{P}$ to reduce the MSE.

\begin{prop}
\label{prop:biasExpected}
When $\mathcal{P}$ is sampled uniformly from $\binom{V(G)}{r,\ldots,r}$, we have
\be
\mathbb{E}_{\mathcal{P}}\left|\mathbb{E}_{\vec{B}}(\xi \mid \mathcal{P}) \right| \le \frac{\bar{K}}{rn-1},
\ee
where $\bar{K}$ is as in \eqref{eqn:kbar}.
\end{prop}

The following result is immediate from Proposition~\ref{prop:biasExpected}.

\begin{cor}
\label{cor:ordinaryCRDbias}
When $\T$ is sampled uniformly from $\binom{V(G)}{pn}$, we have
\be
\left|\mathbb{E}_{\T}(\xi)\right| \le \frac{\bar{K}}{rn-1}.
\ee
\end{cor}
Corollary~\ref{cor:ordinaryCRDbias} generalizes the result of \citet{karwa2016bias} for the case of $f_v(S) = \gamma |S|$ and $p = q = 1$.  As mentioned in Example \ref{ex:lip1}, when $G$ has average degree $m$,  $\bar{K} = \gamma m$. Thus by Corollary~\ref{cor:ordinaryCRDbias}, we obtain that  
$\left|\mathbb{E}_{\T}(\xi)\right| \le \frac{\gamma m}{2n-1}$. 

\section{Symmetric interference model.} \label{sec:SIFmodel}
In this section, we introduce a simple, but natural type of interference function where the interference effect on a vertex depends only on the numbers of its neighbors that are not treated.
Set 
\be
d(V(G)) = \{d(v), v\in V(G)\}
\ee
and define 
\be \label{eqn:SB}
\mathcal{B} = \{(a,b) \in \mathbb{Z}_{\ge 0}^2 \mid a+b \in d(V(G))\}.
\ee
\begin{defn} \label{def: symint} The collection of functions $\{f_v: v \in G\}$  is called a 
\emph{symmetric interference model without types} if there is a function
$f: \mathcal{B} \mapsto \mathbb{R}$  such that
\be \label{eqn:fsymm}
f_v(S) = f(|S|,|\mathcal{N}(v)\setminus S|)
\ee
for all $v \in V(G)$. 
\end{defn}
Thus in Definition \ref{def: symint}, all vertices share the same interference function. 
In the next section, we will allow different \emph{types} of vertices to have different interference functions.

\begin{example}
The family of interference functions $f_v(S) = \gamma |S|$ is achieved in a symmetric interference model without types when $f(a,b) = \gamma a.$
A similar related example is that $f_v(S) = \gamma \frac{|S|}{d(v)}$ is achieved in a symmetric inference model when $f(a,b) = \frac{a}{a+b}$ in Definition \ref{def: symint}.
\end{example}
\begin{example}
In many natural examples, after a certain threshold, adding more number of treated neighbors does not change the interference effect. This can be modeled as $f_v(S) = \gamma\min\{|S|,k\}$ (corresponding to interference only due to the first $k$ treated neighbors) and $f_v(S) = \gamma\min\left\{\frac{|S|}{d(v)}, p/r\right\}$ (corresponding to interference by only the first $p/r$ proportion of treated neighbors). Both of these functions are 
examples of symmetric interference model.
\end{example}

For $\T \subseteq V(G)$ and $v \in V(G)$, let 
\be
\vec{d}_{\T}(v) = (|\T \cap \mathcal{N}(v)|,|\mathcal{N}(v) \setminus \T|)
\ee
denote the \emph{$\T$-bidegree of $v$}.
Let $\Delta^0(\mathcal{B})$ denote the space of finite, signed measures on $\mathcal{B}$ of total mass 0.
When $|\T| = pn$, define the measure $D_\T \in \Delta^0(\mathcal{B})$ as 
\be \label{eqn:measdt}
D_\T(u) = \frac{1}{pqn} \sum_{v \in V} \pq{\T}{v} \delta_{\vec{d}_\T(v)}(u), \quad u \in \mathcal{B}
\ee
where $\mathcal{B}$ is as defined in \eqref{eqn:SB}. Clearly, $D_\T(\mathcal{B}) = \sum_{u \in \mathcal{B}} D_\T(u) = 0$. When the interference function $f_v$ is of symmetric type as in \eqref{eqn:fsymm}, the quantity $\xi$ in Equation \eqref{eqn:xi} can be expressed compactly as 
\be \label{eqn:intxi}
\xi = \int_{\mathcal{B}} f \,dD_\T = \sum_{u \in \mathcal{B}} f(u) D_\T(u).
\ee
\section{Perfect quasi-colorings and designs for symmetric interference model.}
\label{sec:PQC}
In this section, we introduce our idea of perfect quasi-colorings, and use these to construct designs for the symmetric interference model. Throughout this subsection, we will assume that $r=2$ and $p = q = 1,$ so that the target treatment fraction is $\frac{1}{2}$.

The following notion of \emph{perfect quasi-coloring} lets us identify the treatment groups so that the interference effect $\xi$ is identically zero.

\begin{defn}
\label{defn:perfect}
A \emph{perfect quasi-coloring} is a set $Q \in \binom{V(G)}{n}$ that satisfies $D_Q = 0$.
\end{defn}
  
The following result implies that $\xi = 0$ for the the treatment groups $\T = Q$ and $\T = V(G) \setminus Q$ when $Q$ is a perfect quasi-coloring. 
\begin{prop}
\label{prop:perfect}
Let $Q \in \binom{V(G)}{n}$.
The following are equivalent in a symmetric model.
\begin{itemize}
\item $Q$ is a perfect quasi-coloring.
\item $V(G) \setminus Q$ is a perfect quasi-coloring.
\item If $\T= Q$, for every function $f_v$ of the form \eqref{eqn:fsymm}, $\xi = 0$.
\item If $\T = V(G)\setminus Q$, for every function $f_v$ of the form \eqref{eqn:fsymm}, $\xi= 0$.
 \item If the treatment $\T$ is chosen uniformly and randomly between $Q$ and $V(G) \setminus Q$,
 for every function $f_v$ of the form \eqref{eqn:fsymm}, $\xi= 0$.
\end{itemize}
\end{prop}

\begin{remark}
Intuitively, randomizing between $\T = Q$ and $\T = V(G) \setminus Q$ when $Q$ is a perfect quasi-coloring makes $\tnaive$ unbiased because (1) interference effects cancel and (2) each vertex is treated with probability $\frac{1}{2}$, so that each treatment effect enters the estimate with probability $\frac{1}{2}$.
\end{remark}

\begin{proof}
First, we will show that $Q$ is a perfect quasi-coloring then if and only if $V(G) \setminus Q$ is.
Define $\tau: \mathcal{B} \to \mathcal{B}$ by $\tau(a,b) = (b,a)$. Let $\tau_* D_Q$ be the push forward measure of $D_Q$ by the function $\tau$. By construction, it follows that $\tau_* D_Q = -D_{V(G) \setminus Q}.$  Thus we conclude that $D_Q = 0$ if and only if $D_{V(G) \setminus Q} = 0.$

Next, we will prove that $Q$ is a perfect quasi-coloring if and only if $\xi = 0$ for all $f$ when $\T = Q$.
Since, $\xi = \int_{\mathcal{B}} f \, dD$ by Equation \eqref{eqn:intxi},
this assertion follows immediately.
The lemma follows because the distribution $\xi$ with treatment chosen uniformly and randomly between $Q$ and $V(G) \setminus Q$ is a 50--50 mixture of point masses at the values of $\xi$ with $T = Q$ and $T = V(G) \setminus Q$.
\end{proof}

The following example shows that highly homogeneous graphs admit perfect quasi-colorings.

\begin{example}[Perfect quasi-colorings exist in the graph consisting of copies of a smaller graph]
\label{eg:perfectExist}
Let $H$ be an arbitrary graph with $|V(H)| > 1$.
Let $G = H \times \{0,1\}^{V(H)}$.  We claim that
\[Q = \left\{\left(v,(\epsilon_w)_{w \in V(H)}\right) \mid u_v = 1\right\}\]
is a perfect quasi-coloring of $G$.
To see this, define an involution $\psi: V(G) \to V(G)$ by
\[\psi(v,\epsilon)=\left(v,\left(\epsilon_{V(H) \setminus \{v\}},1-\epsilon_v\right)\right).\]
Note that, for all $w \in V(G),$ exactly one of $w$ and $\psi(w)$ is in $Q$ and $w$ and $\psi(w)$ have the same number of neighbors in $Q$ (resp. $V(G) \setminus Q$).
It follows that $D_Q = 0.$
\end{example}

The class of graphs considered by Example~\ref{eg:perfectExist} is quite specific.
Unfortunately, not even $2k$-regular graphs need to admit a perfect quasi-coloring, as the following example shows.

\begin{example}[A hexagon does not have a perfect quasi-coloring]
\label{eg:hexagon}
Let $G$ be a hexagon. Thus $V(G) = \{1,2,\ldots,6\}$ with an edge drawn between $i$ and $i+1$ modulo 6 for all $i$. Let $B \in \binom{V(G)}{3}$.

We claim that $B$ is not perfect.  Indeed, if $B$ contains three consecutive elements of $V(G),$  then the support of $D_B$ contains $(2,0)$.  If $B$ does not contain any three consecutive elements of $V(G),$ then the support of $D_B$ contains $(0,2)$.
In either case, we have $D_B \not= 0$. This example motivates studying other estimators in addition to $\tnaive$; see Discussion for more on this point.
\end{example}

Example~\ref{eg:hexagon} suggests that it might not be fruitful to search for perfect quasi-colorings in arbitrary graphs.  In general, we can only hope to control the size of $\xi$. Proposition \ref{prop:perfect} yields that $\xi = 0$ for a perfect quasi-coloring. It is then natural to ask whether  an ``almost perfect quasi-coloring''  will imply that the corresponding $\xi$ is close to zero. In the next section, we show that this is indeed the case, quantify this intuition and use it constructing new designs. 
\subsection{Quantifying the notion of perfect quasi-coloring.}

Let $\d$ be a metric on $\mathcal{B}$. For $f: \mathcal{B} \mapsto \mathcal{B}$,
define the Lipschitz norm 
\be
\|f\|_{\d} = \sup_{u_1,u_2 \in \mathcal{B}, u_1\neq u_2} {|f(u_1) - f(u_2)| \over \d(u_1, u_2)}.
\ee
For a measure $D \in \Delta^0(\mathcal{B})$, define the \emph{Wasserstein norm}
\be
\|D\|_{\w} = \sup_{\|f\|_{\d} \leq 1} \Big \| \int_{\mathcal{B}} f\, dD \Big \|.
\ee 
Since the total mass is $0$ for any $D \in \Delta^0(\mathcal{B})$, we have that
\be
\|D\|_{\w} \leq \frac{1}{2} \mathrm{diam}(\mathcal{B}) \|D\|_{\TV},
\ee
where $\|-\|_{\TV}$ denotes the total variation norm.
From Equation \eqref{eqn:intxi}, we can deduce that if the interference function $f: \mathcal{B} \mapsto \mathbb{R}$ is Lipschitz with respect to a metric $\d$, then
\be \label{eqn:xiLipbd}
|\xi| \leq \|f\|_{\d}\,\|D_\T\|_{\w}.
\ee
For a treatment assignment $\T$ that is a perfect quasi-coloring, we have $D_\T = 0$
and thus $\xi = 0$. Equation \ref{eqn:xiLipbd} shows that $\xi$ is continuous in $\|D_\T\|_{\w}$.

While \eqref{eqn:xiLipbd} holds for any metric $\d$,  we will use the following 
metric $\d = \d_K$:
\begin{defn}
Fix $K = (K_1,K_2)$ with $K_1 \geq 0$ and $K_2 > 0$, define the metric
$\d_K$ on $\mathcal{B}$: for all $(a,b),(c,d) \in \mathcal{B}$,
\be \label{eqn:nuK}
 \d_{K}((a,b),(c,d)) =  K_1 \frac{|a+b-c-d|}{\dmax} + K_2\big|\frac{a}{a+b} - \frac{c}{c+d}\big|
 \ee 
 where $d_{\max}$ is as in \eqref{eqn:dmin}.
 \end{defn}
\begin{remark}
Since we assume that $G$ does not have any isolated vertices, $\d_{K}$ is indeed a metric on $\mathcal{B}$. 
\end{remark}
\begin{remark}
The choice of a metric is crucial for our estimates. The main point here is that the chosen metric must capture the key features of the interference model. To measure the similarity of two vertices, the metric $\d_K$ in \eqref{eqn:nuK} just takes the differences in the fraction of the treated neighbors and the differences of the degrees between the vertices. This is justified here because, the symmetric interference model by definition depends only on these quantities. Different metrics could be used for other choices of interference functions.
\end{remark}
For $\mathcal{P} = (S_1,\ldots,S_n) \in \binom{V(G)}{r,\ldots,r}$, define the constant 
\be \label{eqn:cp}
\cp = \frac{2}{d_{\max}(r-1)} \sum_{i=1}^n \sum_{\{v,v'\} \subseteq S_i} |d(v)-d(v')|.
\ee

The following proposition bounds the $L^2$ norm of  $\|D_\T\|_{\w}$:
\begin{prop}
\label{prop:TVbound}
Fix $\mathcal{P} \in \binom{V(G)}{r,\ldots,r}$ and let  $\T = \T_{\vec{B},\mathcal{P}}$ 
as in \eqref{eqn:tbp}. We have
\be
\sqrt{\mathbb{E}_{\vec{B}}\left\|D_\T \right\|_{\w}^2} \le \frac{K_1}{\sqrt{pq}n} \cp + \frac{1}{rn}\sum_{v \in V(G)} \frac{4K_2}{\sqrt{d(v)}}+ \frac{1}{pqn} \sum_{v \in V(G)} \frac{|\mathcal{P}_v \cap \mathcal{N}(v)|}{d(v)}
\ee
where $\mathcal{P}_v$ is as in \eqref{eqn:pv}.
\end{prop}

The idea behind the proof of Proposition \ref{prop:TVbound} is to bound the contributions of each vertex to the left-hand-side, and use the fact that $\T \cap S_i$ and $\T \cap S_j$ are independent for $i \not= j,$ where $\mathcal{P}=(S_1,\ldots,S_n)$.

Proposition~\ref{prop:TVbound} and Equation \eqref{eqn:xiLipbd} imply the following upper bound on the $L^2$ norm of $\xi$ for the treatment $\T = \T_{\vec{B},\mathcal{P}}$.

\begin{cor}
\label{cor:MSEgen}
Let the interference function $f:\mathcal{B} \mapsto \mathbb{R}$ be such that
$\|f\|_{\dk} \le 1$.
Then
\be
\sqrt{\mathbb{E}_{\vec{B}}|\xi|^2} \le \frac{K_1}{\sqrt{pq}n} \cp + \frac{1}{rn}\sum_{v \in V(G)} \frac{4K_2}{\sqrt{d(v)}}+ \frac{K_2}{pqn} \sum_{v \in V(G)} \frac{|\mathcal{P}_v \cap \mathcal{N}(v)|}{d(v)}
\ee
for all $\mathcal{P}$ when $\T = \T_{\vec{B},\mathcal{P}}$.
\end{cor}

\begin{remark}
In the case of a complete graph on $rn$ vertices, we have $\cp = 0$ and hence
\be
\sqrt{\mathbb{E}_{\vec{B}}|\xi|^2} \le \frac{4K_2}{\sqrt{rn-1}} + \frac{K_2}{pqn}.
\ee
Thus, for fixed $p,q,r,$ we have $\sqrt{\mathbb{E}_{\vec{B}}|\xi|^2} = O(n^{-1/2})$.
\end{remark}
\section{New Designs and MSE for $\tnaive$.}\label{sec:NeyMSE}
In this section, we will use the idea of perfect quasi-coloring and Proposition \ref{prop:TVbound} to construct new designs for $\tnaive$ and derive bounds for its MSE.
We study the dense ($\dmin \rightarrow \infty$ as $n\rightarrow \infty$)  and sparse 
($\dmin = O(1)$ and $\dmax \rightarrow \infty$ as $n \rightarrow \infty$) cases separately, since our methods and assumptions are different for dense \textit{vs.} sparse graphs.
\subsection{Dense Graphs.}
A key term appearing in the right hand side of Proposition \ref{prop:TVbound} 
is the constant $\cp$, which is solely a function of the partition $\mathcal{P}$. 
Thus we seek for designs which will lead to smaller values for $\mathcal{P}$.
To this end, we introduce the following new design which we call ``partitioning by degree":
\begin{defn} \label{def:pbd}
 Let $\{w^*_i\}, 1 \leq i \leq rn$ be an enumeration of the vertices of $G$ such that 
 $d(w^*_i) \geq d(w^*_{i'})$ whenever $i > {i'}$. Choose 
 \be
 S_i = \{w^*_{j}, (i-1)r + 1\leq j \leq ir\}
 \ee
for $1 \leq i \leq n$. Finally set 
\be
\label{eqn:Pstar}
\mathcal{P}^* = (S_1,\ldots,S_n).
\ee
\end{defn}
Thus the partition $\mathcal{P}^*$ is chosen by first rank ordering the vertices by degree and then pairing vertices of similar degree. 
The following is a key observation:
\begin{lemma} \label{lem:cp}
For $\mathcal{P}^*$ chosen according to partitioning by degree as in 
Definition \ref{def:pbd}, we have
\be \label{eqn:cpstar}
C_{\mathcal{P}^*} \le 2
\ee
where $C_{\mathcal{P}^*}$ in the corresponding constant in Corollary~\ref{cor:MSEgen}.
\end{lemma}
\begin{proof}
Breaking each appearance of $d(v)-d(v')$ in \eqref{eqn:cp} into a sum of terms of the form $d(w^*_k)-d(w^*_{k+1}),$ we have
\be
C_{\mathcal{P}^*} &= \frac{2}{\dmax(r-1)} \sum_{i=1}^n \sum_{1 \le j < j' \le r} \Big(d\big(w^*_{n(i-1)+j}) - d(w^*_{n(i-1)+j'}\big)\Big)\\
&= \frac{2}{\dmax(r-1)} \sum_{i=1}^n \sum_{1 \le j < j' \le r} \sum_{k=j}^{j'-1} \Big(d\big(w^*_{n(i-1)+k}\big) - d\big(w^*_{n(i-1)+k+1}\big)\Big)\\
&\le \frac{2}{\dmax(r-1)} \sum_{i=1}^n \Big[(r-1)\sum_{k=1}^{r-1} \Big(d\big(w^*_{n(i-1)+k}\big) - d\big(w^*_{n(i-1)+k+1}\big)\Big)\Big]\\
&= \frac{2}{\dmax} \sum_{i=1}^n \sum_{k=1}^{r-1} d\big(w^*_{n(i-1)+k}\big) - d\big(w^*_{n(i-1)+k+1}\big)\\
&\le \frac{2}{\dmax} \sum_{k=1}^{nr-1} \left(d(w^*_k) - d(w^*_{k+1})\right) \le \frac{2}{\dmax} \cdot \dmax = 2,
\ee
as desired.
 \end{proof}

As an immediate consequence of  Lemma \ref{lem:cp}, Lemma~\ref{lem:biasGen}, Proposition~\ref{prop:TVbound}, and Corollary~\ref{cor:MSEgen}, we have the following bound:
\begin{theorem}
\label{thm:MSEreg}
When $\T = \T_{\vec{B},\mathcal{P}^*}$,
\be
\sqrt{\mathbb{E}_{\T}\left\|D\right\|^2_{\w}} \le 
\frac{2K_1}{\sqrt{pq}n} + \frac{4K_2}{r\sqrt{\dmin}} + \frac{rK_2\min\{r-1,\dmin\}}{pq\cdot \dmin},
\ee
where $\dmin$ is as in \eqref{eqn:dmin}.
If in addition, the interference function $f$ satisfies $\|f\|_{\dk} \le 1$, then 
\begin{align*}
\left|\mathbb{E}_{\T}\xi\right| &\le \frac{\min\{r-1,\dmin\}(K_1+K_2)}{(r-1)\dmin}\\
\sqrt{\mathbb{E}_{\T}\xi^2} &\le \frac{2K_1}{\sqrt{pq}n} + \frac{4K_2}{r\sqrt{\dmin}} + \frac{rK_2\min\{r-1,\dmin\}}{pq\cdot \dmin}.
\end{align*}
\end{theorem}

Theorem \ref{thm:MSEreg} immediately yields that interference does not affect the consistency of the estimator $\tnaive$ for our randomized design when $G$ grows large and dense.
\begin{cor}
Let $\T = \T_{\vec{B},\mathcal{P}^*}$ and  fix $p,q,r \in \mathbb{N}$. 
If $\dmin \to \infty$ as $n \rightarrow \infty$, then the mean squared error of $\tnaive$ goes to zero as $n \to \infty$.
\end{cor}
The following example shows that the restricted randomization $\T_{\vec{B},\mathcal{P}^*}$, where $\mathcal{P}^*$ is obtained by partitioning by degree as in \eqref{eqn:Pstar}, can significantly outperform the CRD in terms of reducing the mean squared error of $\tnaive$.
\begin{example}
\label{eg:completeIncomplete}
Let $p = q = 1.$
Let $V(G) = \{v_1,\ldots,v_{2k},w_1,\ldots,w_{2k}\}$, and let the edges of $G$ be $\{v_i,v_j\}$.
Thus, $G$ is the disjoint union of a complete graph on $2k$ vertices $V = \{v_1,\ldots,v_{2k}\}$ with $2k$ additional vertices $W = \{w_1,\ldots,w_{2k}\}$.  Consider a symmetric linear interference model $f(a,b) = \gamma a$.

Fix $\T \in \binom{V(G)}{n}$ and let $\alpha = \alpha(T) = |T \cap V|$.
It is straightforward to verify that
\[\xi = \frac{\gamma \left(\alpha (\alpha-1) - (2n-\alpha)\alpha\right)}{2k} = \frac{\gamma}{2n} \alpha(2\alpha - 2n-1).\]
When $\T$ is chosen uniformly and randomly from $\binom{V(G)}{n}$, by the CLT, we have that
 \be
 \frac{\alpha-n}{\sqrt{n}} \to_D \mathcal{N}(0,1/2)
 \ee
as $n \to \infty$.
While $\E_\T \xi \to 0$ as $n \to \infty$, it can be verified using the formulae for higher moments of normal distributions that
\be
(\mathbb{E}_\T|\xi|^2)^{1/2} \sim \gamma\sqrt{n}
\ee
as $n \to \infty$.

On the other hand, note that any $\mathcal{P}^*$ according to \eqref{eqn:Pstar} consists of a partition of $V$ into pairs and a partition of $W$ into pairs.  Therefore, when $\T = \T_{\vec{B},\mathcal{P}^*},$ we have $\alpha = k$ and hence $\xi = -\frac{\gamma}{2}$.\par

Of course in the above example, the graph $G$ contains isolated vertices $\{w_1,\dots, w_{2k}\}$. The conclusions noted above are qualitatively the same if we add some small number of edges between $\{v_1, \cdots, v_{2k}\}$ and $\{w_1,\dots, w_{2k}\}$, with $\dmin \rightarrow \infty$ at a sufficiently slow rate.
\end{example}

Example~\ref{eg:completeIncomplete} illustrates that our treatment design can improve on the completely random design when there is a high degree of heterogeneity in the degrees of vertices.

\subsection{Sparse graphs}
For sparse graphs, the bias bounds implied by Theorem \ref{thm:MSEreg} is a bit weak.
In this setting, it is helpful to randomize over all choices of $\mathcal{P}^*$ in order to reduce bias. To this end, we introduce the following randomized version of the design 
introduced in Definition \ref{def:pbd}:
\begin{defn}\label{def:pbdmod}
Let $S \subseteq V(G)$ be such that no $r$ vertices in $S$ have the same degree and the number of vertices in $V(G) \setminus S$ of each degree is divisible by $r$.
Let $\mathcal{P}^{**}_0$ be sampled uniformly from the set of partitions of $V(G) \setminus S$ into sets of $r$ vertices of the same degree.
Let $S = \{w_1,\ldots,w_{rk}\}$
with
\[d(w_1) \ge d(w_2) \ge \cdots \ge d(w_{rk}).\]
Let \[\mathcal{P}^{**} = \left(\{w_1,\ldots,w_r\},\ldots,\{w_{rk-r+1},\ldots,w_{rk}\},\mathcal{P}^{**}_0\right).\]
\end{defn}
Thus the main difference between designs in Definitions \ref{def:pbd} and \ref{def:pbdmod} is that
in the latter, we randomize over all vertices with same degree instead of merely fixing a partial ordering. Our MSE bounds rely on the following simple observations:
\begin{lemma}
\label{lem:varCorr}
For a sequence random variables $X_1,\ldots,X_k$,
\be
\Var\big(\sum_{i=1}^k X_i\big) \le \sum_{i,j=1}^k \Var(X_i) \left|\Corr(X_i,X_j)\right|.
\ee
\end{lemma}
\begin{proof}
For all $x,y \ge 0$, we have $2\sqrt{xy} \le x+y$. It follows that
\be
2\Cov(X_i,X_j) &= 2\Corr(X_i,X_j) \sqrt{\Var(X_i)}\sqrt{\Var(X_j)}\\
&\le |\Corr(X_i,X_j)| \left(\Var(X_i)+\Var(X_j)\right)
\ee
for all $i,j$.
Thus, we have
\be
\Var\big(\sum_{i=1}^k X_i\big) &= \sum_{i,j} \Cov(X_i,X_j)\\
&\le \frac{1}{2} \sum_{i,j} |\Corr(X_i,X_j)| \left(\Var(X_i)+\Var(X_j)\right)\\
&= \sum_{i,j=1}^k \Var(X_i) \left|\Corr(X_i,X_j)\right|,
\ee
as desired.
\end{proof}

The following simple lemma is crucial to the proof of Proposition~\ref{prop:sparseMSEtypes} below.

\begin{lemma}
\label{lem:varIndpt}
Let $X_1,\ldots,X_n$ be real valued random variables.  Suppose that for each $i,$ there exist at most $\kappa$ indices $j$ such that $X_i$ and $X_j$ are not independent.
Then, we have
\[\Var\Big(\sum_{i=1}^n X_i\Big) \le \kappa \sum_{i=1}^n \Var(X_i).\]
\end{lemma}
\begin{proof}
Since independent random variables are uncorrelated, for each index $i$, there exist at most $\kappa$ indices $j$ such that $\Corr(X_i,X_j) \not= 0$.
It follows that $\sum_{j=1}^k |\Corr(X_i,X_j)| \le \kappa$ for all $i$.
The lemma thus follows from Lemma~\ref{lem:varCorr}.
\end{proof}

Now give the MSE bounds for $\xi$:
\begin{prop}
\label{prop:sparseMSE}
If $\|f\|_{\d} \le 1,$ then
\begin{align*}
\left|\mathbb{E}_\T\xi\right| &\le \frac{K_1}{n} + \frac{3K_2 (\dmax - \dmin)}{n}\\
\sqrt{\mathbb{E}_\T \xi^2} &\le \frac{K_1}{n} + \frac{K_2(\dmax - \dmin)}{n\dmin} + \frac{2K_2\sqrt{\dmax-\dmin}}{\sqrt{n\cdot\dmin}} + \frac{4K_2\sqrt{r^2\dmax^2+1}}{\sqrt{n\cdot\dmin}}\\
&\quad+\frac{rK_2\min\{r-1,\dmin\}\sqrt{r^2\dmax^2+1}}{pq\sqrt{n}\cdot\dmin},
\end{align*}
where $\T = \T_{\vec{B},\mathcal{P}^{**}}$ and $\mathcal{P}^{**}$ is as in Definition \ref{def:pbdmod}.
\end{prop}
Proposition~\ref{prop:sparseMSE} immediately yields the following MSE bounds for $\tnaive$ in the sparse regime:
\begin{cor}
When $\T = \T_{\vec{B},\mathcal{P}^{**}}$, with $\dmax = o(\sqrt{n})$ and $\dmin = O(1)$, the $\mathrm{MSE}$ of $\tnaive$ is $o(1).$
\end{cor}
\section{Interference with types.}\label{sec:inttypes}
In this section we define a generalization of the symmetric interference model discussed in Section \ref{sec:SIFmodel} and derive the MSE bounds for $\tnaive$ under this model.
\begin{defn}\label{def:syminttypes}
 The function $f_v$ is a \emph{symmetric interference model with types} if there exists a partition $\Pi$ of $V(G)$ into sets of even sizes such that there are functions $(f_\pi)_{\pi \in \Pi}$ with
\[f_v(S) = f_{\Pi(v)}(|S|,|\mathcal{N}(v) \setminus S|)\]
for all $v \in V(G)$.  Here, $f_\pi$ is real-valued with domain
\be
\mathcal{B}_\pi = \{(a,b) \in \mathbb{Z}_{\ge 0}^2 \mid a+b \in d(\pi)\}.
\ee
\end{defn}
\begin{remark}
The case of $\Pi =\{V(G)\}$ recovers the symmetric interference model (without types) in Definition \ref{def: symint}.
\end{remark}
Let $\Delta^0(\mathcal{B}_\pi)$ denote the space of finite, signed measures on $\mathcal{B}_{\pi}$ of total mass 0.
When $\pi \subseteq V(G)$ is such that $|\T \cap \pi| = \frac{p|\pi|}{r},$ let
\[\Delta^0(\mathcal{B}_\pi) \ni D^\pi = D^\pi_\T = \frac{1}{pqn} \sum_{v \in \pi} \pq{\T}{v} \delta_{\vec{d}_\T(v)}.\]

\subsection{Perfect quasi-colorings for interference with types.}
The structure of perfect quasi-colorings extends to the setting of interference models with types.
For this subsection, we assume that $p = q = 1,$ so that the target treatment fraction is $\frac{1}{2}$.
The analogue of Definition~\ref{defn:perfect} is:

\begin{defn}
\label{defn:perfectTypes}
A \emph{perfect quasi-coloring of $G$ with respect to the type partition $\Pi$} is a set $B \in \binom{V(G)}{n}$ that satisfies $D^\pi_B = 0$ and $|B \cap \pi| = |\pi|/2$ for all $\pi \in \Pi$.
\end{defn}

Definition~\ref{defn:perfectTypes} recovers Definition~\ref{defn:perfect} by taking $\Pi = \{V(G)\}$.
The analogue of Proposition~\ref{prop:perfect} is:

\begin{prop}
\label{prop:perfectTypes}
Let $B \in \binom{V(G)}{n}$ be such that $|B \cap \pi| = |\pi|/2$ for all $\pi \in \Pi$.
The following are equivalent in a symmetric model.
\begin{itemize}
\item $B$ is a perfect quasi-coloring.
\item $V(G) \setminus B$ is a perfect quasi-coloring.
\item $\xi = 0$ for all $(f_\pi)_{\pi \in \Pi}$ with treatment $T = B$.
\item $\xi = 0$ for all $(f_\pi)_{\pi \in \Pi}$ with treatment $T = V(G) \setminus B$.
\item $\xi = 0$ for all $(f_\pi)_{\pi \in \Pi}$ with treatment chosen uniformly and randomly between $B$ and $V(G) \setminus B$.
\end{itemize}
\end{prop}

The proof of Proposition~\ref{prop:perfectTypes} is similar to the proof of Proposition~\ref{prop:perfect}, as
\[\xi = \sum_{\pi \in \Pi} \int_{\mathcal{B}_\pi} f_\pi \, dD^\pi.\]
Example~\ref{eg:hexagon} shows that perfect quasi-colorings need not exist in general, while the following example generalizes Example~\ref{eg:perfectExist} to exhibit a class of graphs and type partitions in which perfect quasi-colorings exist.

\begin{example}[Perfect quasi-colorings exist in the graph consisting of copies of a smaller graph]
\label{eg:perfectExistTypes}
Let $H$ be an arbitrary graph with $|V(H)| > 1$, and let $\Pi_0$ be a partition of the vertices of $H$.
Let $G = H \times \{0,1\}^{V(H)}$, and define a partition $\Pi$ of $V(G)$ by
\[\Pi = \{\pi \times \{0,1\}^{V(H)} \mid \pi \in \Pi_0\}.\]
It is straightforward to verify that
\[B = \left\{v,(\epsilon_w)_{w \in V(H)} \mid u_v = 1\right\}\]
is a perfect quasi-coloring of $G$ with respect to the type partition $\Pi$.
\end{example}

\subsection{Semi-restricted randomization}

For each $\pi \in \Pi,$ let $\T_\pi$ be drawn uniformly and randomly from $\binom{\pi}{p|\pi|/r},$ with $(\T_\pi)_{\pi \in \Pi}$ independent.
Define $T_\Pi = \bigcup_{\pi \in \Pi} \T_\pi.$

We can represent this treatment group in terms of a restricted randomization treatment group as follows.
Let $\mathcal{P}$ be sampled uniformly from $\binom{\Pi}{r,\ldots,r}$ (independently of $\vec{B}$).
Then, $\T_{\vec{B},\mathcal{P}}$ has the same distribution as $\T_\Pi$.

\subsection{MSE bounds}
In this section, we will use the following 
metric $\d$:
\begin{defn}
Fix $K > 0$, define the metric
$\d$ on $\mathcal{B}_\pi$: for all $(a,b),(c,d) \in \mathcal{B}_{\pi}$,
\be \label{eqn:nuK}
 \d((a,b),(c,d)) = K_\big|\frac{a}{a+b} - \frac{c}{c+d}\big|
 \ee 
 where $d_{\max}$ is as in \eqref{eqn:dmin}.
 \end{defn}
The analogue of Proposition~\ref{prop:TVbound} in this setting is:

\begin{prop}
\label{prop:TVboundTypes}
For all $\mathcal{P} \in \binom{\Pi}{r,\ldots,r},$ we have
\[\sqrt{\mathbb{E}_{\vec{B}}\Big[\big(\sum_{\pi \in \Pi} \|D^\pi\|_{\w}\big)^2\Big]} \le \frac{1}{rn}\sum_{v \in V(G)} \frac{4K}{\sqrt{d(v)}} + \frac{K}{pqn} \sum_{v \in V(G)} \frac{|\mathcal{P}_v \cap \mathcal{N}(v)|}{d(v)}\]
when $\T = \T_{\vec{B},\mathcal{P}}$.  If $\|f_\pi\|_{\d} \le 1$ for all $\pi \in \Pi,$ then
\[\sqrt{\mathbb{E}_\T\xi^2} \le \frac{1}{rn}\sum_{v \in V(G)} \frac{4K}{\sqrt{d(v)}} + \frac{K}{pqn} \sum_{v \in V(G)} \frac{|\mathcal{P}_v \cap \mathcal{N}(v)|}{d(v)}\]
when $\T = \T_{\vec{B},\mathcal{P}},$ hence also when $\T = \T_\Pi$.
\end{prop}

The analogue of Proposition~\ref{prop:sparseMSE} is:

\begin{prop}
\label{prop:sparseMSEtypes}
If $\|f_\pi\|_{\dkp} \le 1$ for all $\pi \in \Pi,$ then
\begin{align*}
\left|\mathbb{E}_\T \xi\right| &\le \frac{K|\Pi|}{rn}\\
\frac{\sqrt{\mathbb{E}_\T\xi^2}}{K} &\le \frac{\sqrt{2|\Pi|}}{\sqrt{nr\cdot\dmin}} + \frac{4\sqrt{r^2\dmax^2+1}}{\sqrt{n\cdot\dmin}}+\frac{r\min\{r-1,\dmin\}\sqrt{r^2\dmax^2+1}}{pq\sqrt{n}\cdot\dmin}
\end{align*}
when $\T = \T_\Pi$.
\end{prop}

Thus, in the sparse setting, it is important that $\Pi$ is not too large, \textit{i.e.}, that there are not too many different types of vertices.
The condition that $|\Pi|$ not be too large is analogous to the condition that $\dmax - \dmin$ not be too large implicit in the statement of Proposition~\ref{prop:sparseMSE}.

\section{Homophily and types.}\label{sec:homophily}
In this section, we directly bound the MSE of $\tnaive$ in a model that allows homophily between vertices in a single element of $\Pi$. 

For $\pi \in \Pi,$ let
\be
x_\pi &= \frac{1}{|\pi|} \sum_{v \in \pi} x_v,\, \quad t_\pi = \frac{1}{|\pi|} \sum_{v \in \pi} t_v,
\ee
be the average covariate effect and average treatment effect respectively within a type. If homophily is suspected, one expects that $x_v$ will be close to $x_\pi$ and $t_v$ will be close to $t_\pi$ within a type.
To that end for $v \in V(G),$ let
\begin{align*}
\epsilon_v &= x_v - x_{\Pi(v)} + \frac{q}{r}\left(t_v - t_{\Pi(v)}\right)
\end{align*}
be the discrepancy between an individual node's behavior and their type average.
Then
\[\sigma^2 = \frac{1}{rn} \sum_{v \in V(G)} \epsilon_v^2\] captures the sum of squared differences between nodes and their type averages within a graph. Thus $\sigma^2$ has an inverse relationship with homophily. 
The following result, which generalizes Lemma~\ref{lem:idtUnbiased}, bounds the MSE of $\idt$.

\begin{prop}
\label{prop:knownTypes}
For all partitions $\Pi$ of $V(G)$ into sets of size divisible by $r$, we have
\be
\mathbb{E}_\T\idt &= \bar{t}\\
\Var_\T\left(\idt\right) &\le \frac{2r\sigma^2}{pqn}
\ee
when $\T = \T_\Pi$.
\end{prop}

Coupling Proposition~\ref{prop:knownTypes} with bounds on $\xi$ yields the following bias and MSE bounds for $\tnaive$ when there is homophily.

\begin{cor}
\label{cor:knownTypes}
If $\|f_\pi\|_{\dkp} \le 1$ for all $\pi \in \Pi,$ then
\begin{align*}
\left|\mathbb{E}_\T\tnaive - \bar{t} \right| &\le \frac{K|\Pi|}{n}\\
\sqrt{\mathbb{E}_\T\left(\tnaive - \bar{t}\right)^2} &\le \frac{1}{rn}\sum_{v \in V(G)} \frac{4K}{\sqrt{d(v)}} + \frac{K}{pqn} \sum_{v \in V(G)} \frac{(r-1)|\Pi_v \cap \mathcal{N}(v)|}{(|\Pi_v|-1) \cdot d(v)} + \frac{\sigma\sqrt{2r}}{\sqrt{pqn}} 
\end{align*}
when $\T = \T_\Pi$.
\end{cor}
\begin{proof}
Follows from Propositions~\ref{prop:TVboundTypes},~\ref{prop:sparseMSEtypes}, and~\ref{prop:knownTypes}.
\end{proof}

The results of this section are closely related to the work of \citet{basse2015optimal} on optimal design with network correlated outcomes that are induced by homophily but no interference.

\section{Simulations.}\label{sec:sims}
In this section we conduct a series of simulations to demonstrate the efficacy of the approach. We vary the following parameters: the type of the network and the strength of the interference.
For each of the simulations we consider the following model:
\[y_v = x_v + t_v1_{v\in T} + f_v(\T \cap \mathcal{N}(v))\] where $x_v\overset{iid}{\sim} \mathrm{N}(0,1)$ and $t_v\overset{iid}{\sim}\mathrm{N}(2,0.25)$. That is, the baseline outcome for all of the individuals in the graph is centered at 0 with a variance of 1, while the treatment effect for everyone is centered at 2 with a variance of 0.25. We consider two treatment regimes: our approach (described in Section \ref{sec:NeyMSE}) and the completely randomized design where exactly half of all units are treated randomly. We report log mean squared errors (log MSE) for the Neymanian estimator in Figure~\ref{fig:graphs} for the two sets of simulations. The MSEs are calculated over 10000 simulated randomizations for each approach.

\begin{figure}[t]
	\subcaptionbox{Erdos-Renyi graphs with linear interference\label{subf:er}}{\includegraphics[width=\textwidth]{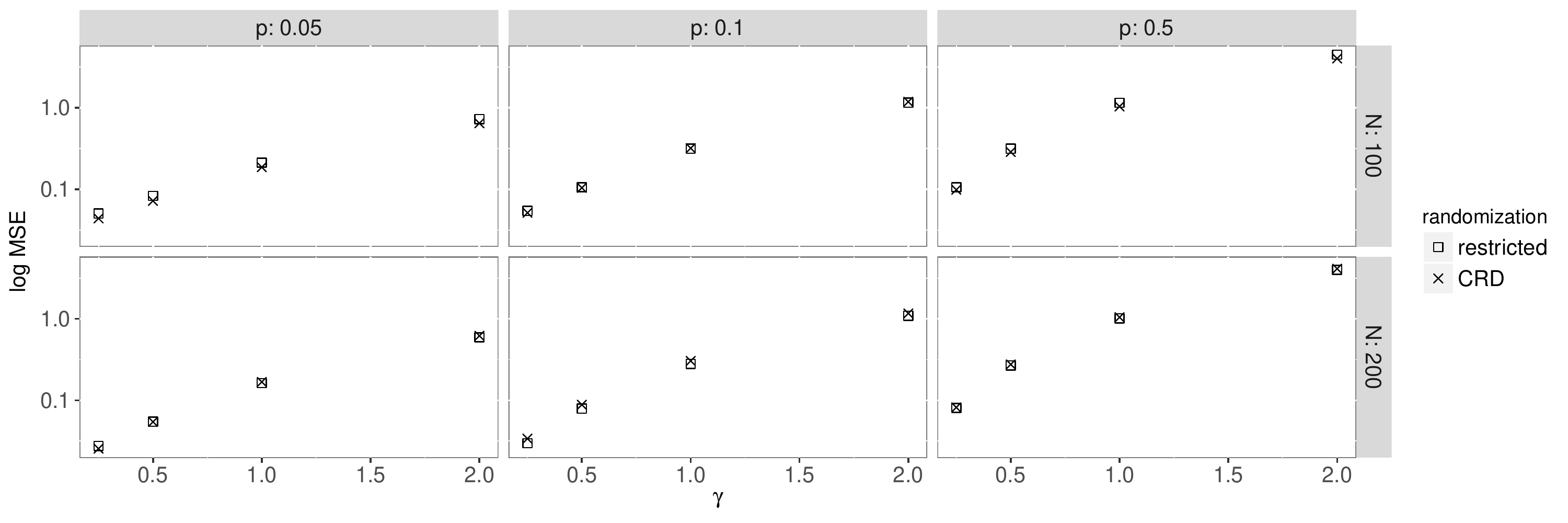}}
	\subcaptionbox{Preferential attachment graphs with linear interference\label{subf:pa_lin}}{\includegraphics[width=\textwidth]{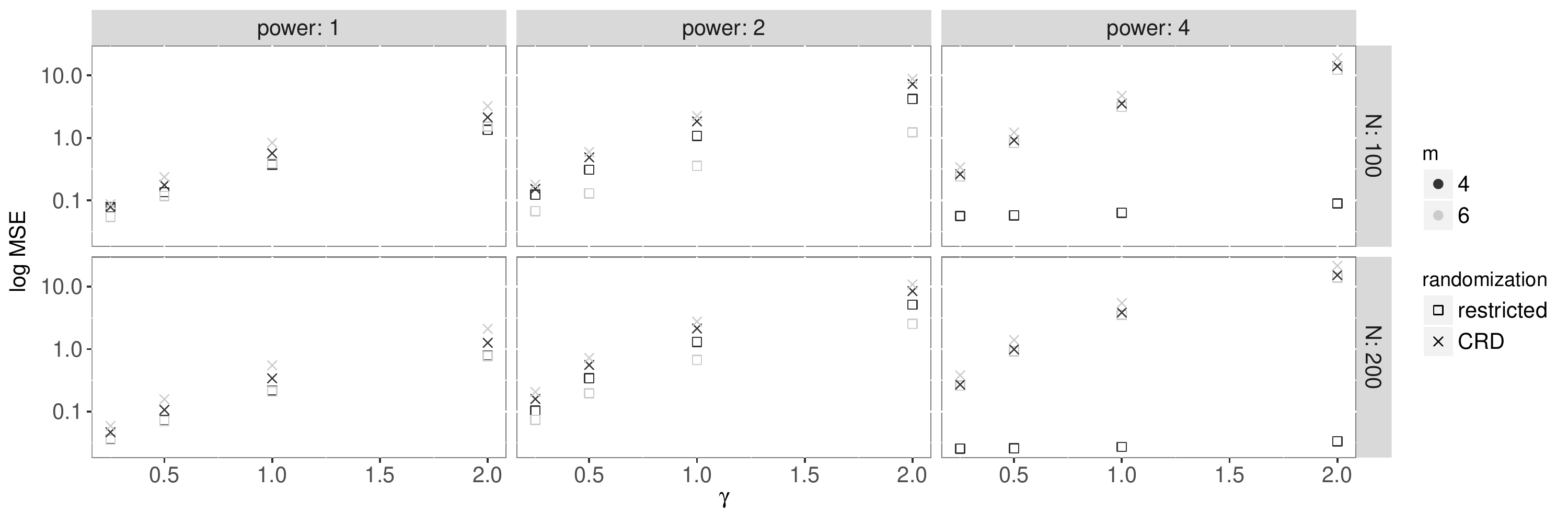}}
	\caption{Simulation results}
	\label{fig:graphs}
\end{figure}
	
\subsection{Erdos-Renyi graphs.} 
In this simulation we generate a graph $G\sim \text{ER}(N,p)$ with $N$ nodes and overall density $p$. This is an independent edge random graph model where an edge between node $v$ and $v^\prime$ exists with probability $p$. We consider the symmetric linear interference function $f_v(A) = \gamma|A|$. We consider three graph sizes: $100$, $200$ and four graph densities: $0.05$, $0.1$, $0.5$. The parameter $\gamma$ varies from $0.1$ to $2$. 

An important quality of Erdos-Renyi graphs is that they are extremely dense (expected degree is $Np$) and the degrees of their nodes concentrate \citep{lu2012spectra}. Further, this model gives rise to large cliques within the graph implying that many nodes have the exact same degree and are connected to each other \citep{bollobas1976cliques}. Because of these traits a randomization scheme based on the degree distribution of the graph is unlikely to perform well. In fact, our proposed procedure behaves similarly to the standard Bernoulli randomization scheme. This behavior is evident in Figure~\ref{subf:er} where both estimators have approximately the same log MSE with the CRD even exhibiting better behavior for denser graphs and higher levels of interference (such as $p=0.5,\ \gamma=2$). 

\subsection{Preferential attachment graphs.} 
In this simulation we generate a graph $G\sim \mathrm{PA}(N,\pow,m)$ with $N$ nodes, $\pow$ power of the preferential attachment (PA) and $m$ new edges at each step of the graph growth \citep{barabasi1999emergence}. These graphs are constructed by staring with a single vertex and adding 1 new vertex at a time. The new vertex forms an edge with an existing vertex $v$ with probability $d(v)^{\pow}$. Each new vertex forms $m$ new edges. This process continues until there are $N$ vertices in the graph. These graphs have power law degree distributions and hence are sparse with many small degree nodes and a few large hubs. 

It is clear that log MSE increases with power since it produces denser graphs that are more likely to have too many nodes with the same degree. However, the behavior with respect to $m$ is more complicated. In Figure~\ref{subf:pa_lin} we see the log MSE of the estimator based on CRD increase in $m$ for all levels of 
$\pow$. However, this is not the case for the restricted randomization. This behavior is likely explained by the special behavior of super-linear preferential attachment \citep{krapivsky2000connectivity}. When $\pow=4$ and $m=4$ most graphs have four central nodes that are connected to everyone else. As such, only these central nodes induce any form of interference on the other nodes and so the restricted randomization ideally allocated treatment. The CRD does not take this structure into account and so frequently is likely to allocate all of the central nodes to treatment or control, leading to increased bias and variance. When $m=6$ there are enough perturbations in the system to lead to poorer performance by the restricted randomization.
On the other hand, when $\pow\in (1,2]$, small $m$ frequently lead to the creation of an odd number of central nodes while a large $m$ produces a large amount of heterogeneity in the degrees. In this setting, the restricted randomization approach prefers more heterogeneity as it balances the interference among nodes. In all of these settings, the CRD performs worse than the restricted randomization.



\section{Discussion.}
This article provided a new approach to bounding the bias and mean squared error of the Neymanian estimator of the average treatment effect under interference and homophily. It introduced the notion of quasi-coloring to better understand the balance needed in the randomization scheme to account for interference. Based on this construct we developed a restricted randomization scheme that has good theoretical properties and performs well in simulations. There are a number of directions for future research.

The general notion of perfect quasi-coloring provides an intuition for constructing other linear unbiased estimators. For example, we can construct a partial-perfect-quasi-coloring by only treating one node. This produces the following unbiased estimator: $Y_{\mathrm{treated}}- \bar{Y}_c$,
where $\bar{Y}_c$ is the average outcome of all the control units who are not neighbors of the treated unit.
The weights associated with treated and control units are still interpretable.\par It is also possible to develop the machinery in this paper for other estimands and estimators of interest. However, this requires even greater care. For example, we could be interested in the interference effect of exactly one treated neighbor --- this lends itself naturally to specifying several naive Neyman-type estimators: only consider control (treated) nodes who have one treated neighbor versus control (treated) nodes who have no treated neighbors, or some combination of both. In turn, this suggests particular restrictions on the randomization scheme. More general versions of this approach can be studied for less constrained types of interference.
\section*{Acknowledgements.}The authors thank Edoardo Airoldi, Dean Eckles, Vishesh Karwa and Daniel Sussman for helpful conversations.
Part of this research was conducted while RJ was an Economic Design Fellow at the Harvard Center of Mathematical Sciences and Applications.
NSP was partially supported by an ONR grant. AV was partially supported by a NSF MSPRF.
\bibliographystyle{imsart-nameyear}
\bibliography{biblio}

\begin{thebibliography}{23}

\bibitem[\protect\citeauthoryear{Aiello et~al.}{2016}]{aiello2016design}
\begin{barticle}[author]
\bauthor{\bsnm{Aiello},~\bfnm{Allison~E}\binits{A.~E.}},
  \bauthor{\bsnm{Simanek},~\bfnm{Amanda~M}\binits{A.~M.}},
  \bauthor{\bsnm{Eisenberg},~\bfnm{Marisa~C}\binits{M.~C.}},
  \bauthor{\bsnm{Walsh},~\bfnm{Alison~R}\binits{A.~R.}},
  \bauthor{\bsnm{Davis},~\bfnm{Brian}\binits{B.}},
  \bauthor{\bsnm{Volz},~\bfnm{Erik}\binits{E.}},
  \bauthor{\bsnm{Cheng},~\bfnm{Caroline}\binits{C.}},
  \bauthor{\bsnm{Rainey},~\bfnm{Jeanette~J}\binits{J.~J.}},
  \bauthor{\bsnm{Uzicanin},~\bfnm{Amra}\binits{A.}},
  \bauthor{\bsnm{Gao},~\bfnm{Hongjiang}\binits{H.}} \betal{et~al.}
(\byear{2016}).
\btitle{Design and methods of a social network isolation study for reducing
  respiratory infection transmission: The eX-FLU cluster randomized trial}.
\bjournal{Epidemics}
\bvolume{15}
\bpages{38--55}.
\end{barticle}
\endbibitem

\bibitem[\protect\citeauthoryear{Barab{\'a}si and
  Albert}{1999}]{barabasi1999emergence}
\begin{barticle}[author]
\bauthor{\bsnm{Barab{\'a}si},~\bfnm{Albert-L{\'a}szl{\'o}}\binits{A.-L.}} \AND
  \bauthor{\bsnm{Albert},~\bfnm{R{\'e}ka}\binits{R.}}
(\byear{1999}).
\btitle{Emergence of scaling in random networks}.
\bjournal{Science}
\bvolume{286}
\bpages{509--512}.
\end{barticle}
\endbibitem

\bibitem[\protect\citeauthoryear{Basse and Airoldi}{2015}]{basse2015optimal}
\begin{barticle}[author]
\bauthor{\bsnm{Basse},~\bfnm{Guillaume~W}\binits{G.~W.}} \AND
  \bauthor{\bsnm{Airoldi},~\bfnm{Edoardo~M}\binits{E.~M.}}
(\byear{2015}).
\btitle{Optimal design of experiments in the presence of network-correlated
  outcomes}.
\bjournal{ArXiv e-prints}.
\end{barticle}
\endbibitem

\bibitem[\protect\citeauthoryear{Bollob{\'a}s and
  Erd{\"o}s}{1976}]{bollobas1976cliques}
\begin{binproceedings}[author]
\bauthor{\bsnm{Bollob{\'a}s},~\bfnm{B{\'e}la}\binits{B.}} \AND
  \bauthor{\bsnm{Erd{\"o}s},~\bfnm{Paul}\binits{P.}}
(\byear{1976}).
\btitle{Cliques in random graphs}.
In \bbooktitle{Mathematical Proceedings of the Cambridge Philosophical Society}
\bvolume{80}
\bpages{419--427}.
\bpublisher{Cambridge Univ Press}.
\end{binproceedings}
\endbibitem

\bibitem[\protect\citeauthoryear{Choi}{2016}]{choi2016estimation}
\begin{barticle}[author]
\bauthor{\bsnm{Choi},~\bfnm{David}\binits{D.}}
(\byear{2016}).
\btitle{Estimation of monotone treatment effects in network experiments}.
\bjournal{Journal of the American Statistical Association}
\banumber{just-accepted}.
\end{barticle}
\endbibitem

\bibitem[\protect\citeauthoryear{Eckles, Karrer and
  Ugander}{2017}]{eckles2017design}
\begin{barticle}[author]
\bauthor{\bsnm{Eckles},~\bfnm{Dean}\binits{D.}},
  \bauthor{\bsnm{Karrer},~\bfnm{Brian}\binits{B.}} \AND
  \bauthor{\bsnm{Ugander},~\bfnm{Johan}\binits{J.}}
(\byear{2017}).
\btitle{Design and Analysis of Experiments in Networks: Reducing Bias from
  Interference}.
\bjournal{Journal of Causal Inference}
\bvolume{5}.
\end{barticle}
\endbibitem

\bibitem[\protect\citeauthoryear{Gruhl et~al.}{2004}]{gruhl2004information}
\begin{binproceedings}[author]
\bauthor{\bsnm{Gruhl},~\bfnm{Daniel}\binits{D.}},
  \bauthor{\bsnm{Guha},~\bfnm{Ramanathan}\binits{R.}},
  \bauthor{\bsnm{Liben-Nowell},~\bfnm{David}\binits{D.}} \AND
  \bauthor{\bsnm{Tomkins},~\bfnm{Andrew}\binits{A.}}
(\byear{2004}).
\btitle{Information diffusion through blogspace}.
In \bbooktitle{Proceedings of the 13th International Conference on World Wide
  Web}
\bpages{491--501}.
\bpublisher{ACM}.
\end{binproceedings}
\endbibitem

\bibitem[\protect\citeauthoryear{Hoff, Raftery and
  Handcock}{2002}]{hoff2002latent}
\begin{barticle}[author]
\bauthor{\bsnm{Hoff},~\bfnm{Peter~D.}\binits{P.~D.}},
  \bauthor{\bsnm{Raftery},~\bfnm{Adrian~E.}\binits{A.~E.}} \AND
  \bauthor{\bsnm{Handcock},~\bfnm{Mark~S.}\binits{M.~S.}}
(\byear{2002}).
\btitle{Latent space approaches to social network analysis}.
\bjournal{Journal of the American Statistical Association}
\bvolume{97}
\bpages{1090--1098}.
\end{barticle}
\endbibitem

\bibitem[\protect\citeauthoryear{Holland}{1986}]{holland1986statistics}
\begin{barticle}[author]
\bauthor{\bsnm{Holland},~\bfnm{Paul~W.}\binits{P.~W.}}
(\byear{1986}).
\btitle{Statistics and causal inference}.
\bjournal{Journal of the American Statistical Association}
\bvolume{81}
\bpages{945--960}.
\end{barticle}
\endbibitem

\bibitem[\protect\citeauthoryear{Holland, Laskey and
  Leinhardt}{1983}]{holland1983stochastic}
\begin{barticle}[author]
\bauthor{\bsnm{Holland},~\bfnm{Paul~W}\binits{P.~W.}},
  \bauthor{\bsnm{Laskey},~\bfnm{Kathryn~Blackmond}\binits{K.~B.}} \AND
  \bauthor{\bsnm{Leinhardt},~\bfnm{Samuel}\binits{S.}}
(\byear{1983}).
\btitle{Stochastic blockmodels: First steps}.
\bjournal{Social Networks}
\bvolume{5}
\bpages{109--137}.
\end{barticle}
\endbibitem

\bibitem[\protect\citeauthoryear{Hudgens and
  Halloran}{2008}]{hudgens2008toward}
\begin{barticle}[author]
\bauthor{\bsnm{Hudgens},~\bfnm{Michael~G}\binits{M.~G.}} \AND
  \bauthor{\bsnm{Halloran},~\bfnm{M~Elizabeth}\binits{M.~E.}}
(\byear{2008}).
\btitle{Toward causal inference with interference}.
\bjournal{Journal of the American Statistical Association}
\bvolume{103}
\bpages{832--842}.
\end{barticle}
\endbibitem

\bibitem[\protect\citeauthoryear{Karwa and Airoldi}{2016}]{karwa2016bias}
\begin{barticle}[author]
\bauthor{\bsnm{Karwa},~\bfnm{Vishesh}\binits{V.}} \AND
  \bauthor{\bsnm{Airoldi},~\bfnm{Edoardo}\binits{E.}}
(\byear{2016}).
\btitle{Bias of classic estimates under interference.}
\bjournal{In preparation}.
\end{barticle}
\endbibitem

\bibitem[\protect\citeauthoryear{Krapivsky, Redner and
  Leyvraz}{2000}]{krapivsky2000connectivity}
\begin{barticle}[author]
\bauthor{\bsnm{Krapivsky},~\bfnm{Paul~L.}\binits{P.~L.}},
  \bauthor{\bsnm{Redner},~\bfnm{Sidney}\binits{S.}} \AND
  \bauthor{\bsnm{Leyvraz},~\bfnm{Francois}\binits{F.}}
(\byear{2000}).
\btitle{Connectivity of growing random networks}.
\bjournal{Physical Review Letters}
\bvolume{85}
\bpages{4629}.
\end{barticle}
\endbibitem

\bibitem[\protect\citeauthoryear{Lu and Peng}{2012}]{lu2012spectra}
\begin{barticle}[author]
\bauthor{\bsnm{Lu},~\bfnm{Linyuan}\binits{L.}} \AND
  \bauthor{\bsnm{Peng},~\bfnm{Xing}\binits{X.}}
(\byear{2012}).
\btitle{Spectra of edge-independent random graphs}.
\bjournal{arXiv preprint arXiv:1204.6207}.
\end{barticle}
\endbibitem

\bibitem[\protect\citeauthoryear{Moody}{2001}]{moody2001race}
\begin{barticle}[author]
\bauthor{\bsnm{Moody},~\bfnm{James}\binits{J.}}
(\byear{2001}).
\btitle{Race, school integration, and friendship segregation in America}.
\bjournal{American Journal of Sociology}
\bvolume{107}
\bpages{679--716}.
\end{barticle}
\endbibitem

\bibitem[\protect\citeauthoryear{Rubin}{1974}]{rubin1974estimating}
\begin{barticle}[author]
\bauthor{\bsnm{Rubin},~\bfnm{Donald~B.}\binits{D.~B.}}
(\byear{1974}).
\btitle{Estimating causal effects of treatments in randomized and nonrandomized
  studies.}
\bjournal{Journal of Educational Psychology}
\bvolume{66}
\bpages{688}.
\end{barticle}
\endbibitem

\bibitem[\protect\citeauthoryear{Rubin}{1978}]{rubin1978bayesian}
\begin{barticle}[author]
\bauthor{\bsnm{Rubin},~\bfnm{Donald~B.}\binits{D.~B.}}
(\byear{1978}).
\btitle{Bayesian inference for causal effects: The role of randomization}.
\bjournal{Annals of Statistics}
\bpages{34--58}.
\end{barticle}
\endbibitem

\bibitem[\protect\citeauthoryear{Rubin}{2008}]{rubin2008objective}
\begin{barticle}[author]
\bauthor{\bsnm{Rubin},~\bfnm{Donald~B.}\binits{D.~B.}}
(\byear{2008}).
\btitle{For objective causal inference, design trumps analysis}.
\bjournal{Annals of Applied Statistics}
\bpages{808--840}.
\end{barticle}
\endbibitem

\bibitem[\protect\citeauthoryear{Shalizi and
  Thomas}{2011}]{shalizi2011homophily}
\begin{barticle}[author]
\bauthor{\bsnm{Shalizi},~\bfnm{Cosma~Rohilla}\binits{C.~R.}} \AND
  \bauthor{\bsnm{Thomas},~\bfnm{Andrew~C}\binits{A.~C.}}
(\byear{2011}).
\btitle{Homophily and contagion are generically confounded in observational
  social network studies}.
\bjournal{Sociological Methods \& Research}
\bvolume{40}
\bpages{211--239}.
\end{barticle}
\endbibitem

\bibitem[\protect\citeauthoryear{Sussman and
  Airoldi}{2017}]{sussman2017elements}
\begin{barticle}[author]
\bauthor{\bsnm{Sussman},~\bfnm{Daniel~L.}\binits{D.~L.}} \AND
  \bauthor{\bsnm{Airoldi},~\bfnm{Edoardo~M.}\binits{E.~M.}}
(\byear{2017}).
\btitle{Elements of estimation theory for causal effects in the presence of
  network interference}.
\bjournal{arXiv preprint arXiv:1702.03578}.
\end{barticle}
\endbibitem

\bibitem[\protect\citeauthoryear{Toulis and Kao}{2013}]{toulis2013estimation}
\begin{binproceedings}[author]
\bauthor{\bsnm{Toulis},~\bfnm{Panos}\binits{P.}} \AND
  \bauthor{\bsnm{Kao},~\bfnm{Edward}\binits{E.}}
(\byear{2013}).
\btitle{Estimation of causal peer influence effects}.
In \bbooktitle{International Conference on Machine Learning}
\bpages{1489--1497}.
\end{binproceedings}
\endbibitem

\bibitem[\protect\citeauthoryear{Toulis, Volfovsky and
  Airoldi}{2017}]{toulis2017elements}
\begin{barticle}[author]
\bauthor{\bsnm{Toulis},~\bfnm{Panos}\binits{P.}},
  \bauthor{\bsnm{Volfovsky},~\bfnm{Alexander}\binits{A.}} \AND
  \bauthor{\bsnm{Airoldi},~\bfnm{Edoardo~M.}\binits{E.~M.}}
(\byear{2017}).
\btitle{Causal inference in the presence of entangled treatments in
  observational studies}.
\bjournal{In preparation}.
\end{barticle}
\endbibitem

\bibitem[\protect\citeauthoryear{Yang and Counts}{2010}]{yang2010predicting}
\begin{barticle}[author]
\bauthor{\bsnm{Yang},~\bfnm{Jiang}\binits{J.}} \AND
  \bauthor{\bsnm{Counts},~\bfnm{Scott}\binits{S.}}
(\byear{2010}).
\btitle{Predicting the Speed, Scale, and Range of Information Diffusion in
  Twitter.}
\bjournal{ICWSM}
\bvolume{10}
\bpages{355--358}.
\end{barticle}
\endbibitem

\end{thebibliography}
\appendix
\section{Bounds on bias}

For $v \in V(G)$ and $\T \subseteq \binom{V(G)}{n},$ let
\be
\xi_v = \pq{\T}{v} f_v(\T \cap \mathcal{N}(V))
\ee
denote the interference effect on $v$, 
so that $\xi$ in \eqref{eqn:xi} can be expressed as
\be
\xi = \frac{1}{pqn} \sum_{v \in V(G)} \xi_v.
\ee

Given $v,w \in V(G)$, define the \emph{weight of $w$ on $v$} as
\be[eqn:wt]
W_v(w) &=  \sup_{A \subseteq \mathcal{N}(v) \setminus \{w\}} \left|f_v(A)- f_v(A \cup \{w\})\right|, \quad w \in \mathcal{N}(v) \\
&= 0, \quad \mathrm{otherwise}.
\ee

\begin{lemma}
\label{lem:biasWeight}
For a partition $\mathcal{P} = \{S_1,\ldots,S_n\} \in \binom{V(G)}{r,\ldots,r}$ and the treatment
assignment mechanism $\T_{\vec{B},\mathcal{P}}$ given in \eqref{eqn:tbp}, we have
\be
\left|\mathbb{E}_{\vec{B}}(\xi )\right| \le \frac{1}{nr(r-1)} \sum_{i=1}^n \sum_{\{w,w'\} \in \binom{S_i}{2}} \big(W_w(w')+W_{w'}(w)\big).
\ee
\end{lemma}

The full generality of Lemma~\ref{lem:biasWeight} may be of use in a weighted interference model, as the formalism of weights allows one to capture the fact that different connections may have different strengths.
Including a weak connection (low weight edge) in $\mathcal{P}$ will affect the bias less than including a strong connection.
The following result will imply Lemma~\ref{lem:biasWeight}.
\begin{lemma}
\label{lem:biasPair}
For all $j=1,\ldots,n$ and $v \in S_j,$ we have
\be
\left |\mathbb{E}_{B_j}\left[\xi_v \mid B_1,\ldots,B_{j-1},B_{j+1},\ldots,B_n\right]\right | \le \frac{pq}{r(r-1)}\sum_{w \in S_j \setminus \{v\}} W_v(w)
\ee
when $\T = \T_{\vec{B},\mathcal{P}}.$
\end{lemma}
\begin{proof}
Without loss of generality, assume that $j = 1$ and $v = w^1_1$.
When $1 \in B_1,$ define a random variable $B'_1$ with values in $\binom{[r]}{p}$ by choosing $B'_1$ uniformly from
\[\left\{A \in \binom{[r]}{p} \mid B_1 \setminus A = \{1\}\right\}.\]

Let $B'_i = B_i$ for $i\not= 1.$
Denote by $\xi'$ (resp. $\xi'_v$) the interference effect $\xi$ (resp. the interference effect $\xi'_v$ on $v$) for the treatment group $\T' = \T_{\vec{B}',\mathcal{P}}$.

When $1 \in B_1,$ we have
\be
p\xi_v + q\xi'_v &= pq\left(f_v(\T \cap \mathcal{N}(v)) + f_v(\T' \cap \mathcal{N}(v))\right).
\ee
Thus, when $1 \in B_1,$ we have
\be
\left|p\xi_v + q\xi'_v\right| &\le pq\left|f_v\left(\T \cap \mathcal{N}(v)\right) - f_v\left(\T' \cap \mathcal{N}(v)\right)\right| \le W_v(w),
\ee
where $\T \Delta \T' = \{v,w\}.$
Taking expectations with respect to $B'_1$, it follows that, when $1 \in B_1,$ we have
\be
\mathbb{E}_{B'_1}\left[\left|p\xi_v + q\xi'_v\right| \mid \vec{B}\right] \le \frac{pq}{r-1} \sum_{w \in S_1 \setminus \{v\}} W_v(w).
\ee
By the triangle inequality, we have
\be\label{eqn:xixip}
\left|\mathbb{E}_{B_1}\left[\frac{p\xi_v+q\xi'_v}{r} \mid 1 \notin B_1 \mid B_2,\ldots,B_n\right]\right | \le \frac{pq}{r(r-1)} \sum_{w \in S_1 \setminus \{v\}} W_v(w)
\ee

Note that $\mathcal{L}(B_1 \mid 1 \notin B_1) = \mathcal{L}(B'_1 \mid 1 \in B_1),$ where $\mathcal{L}$ denotes the law of a random variable.
It follows that
\be\label{eqn:xisym}
&\mathbb{E}_{B_j}\left[\xi_v \mid 1 \notin B_1 \mid B_2,\ldots,B_n\right]
= \mathbb{E}_{B'_1}\left[\xi'_v \mid 1 \in B_1 \mid B_2,\ldots,B_n\right].
\ee
Combining (\ref{eqn:xixip}) and (\ref{eqn:xisym}) and using the fact that \[\Pr[1 \in B_1 \mid B_2,\ldots,B_n] = \frac{p}{r},\]
we obtain the lemma.
\end{proof}

\begin{proof}[Proof of Lemma~\ref{lem:biasWeight}]
It follows from Lemma~\ref{lem:biasPair} that
\[\left|\E_{\vec{B}} \xi_v\right| \le \frac{pq}{r(r-1)} \sum_{w \in S_j \setminus \{v\}} W_v(w).\]
Summing over $v \in V(G),$ we have
\[\left|\E_{\vec{B}} \xi\right| \le \frac{1}{pqn} \sum_{v \in V(G)} \left|\E_{\vec{B}} \xi_v\right| \le \frac{1}{nr(r-1)} \sum_{i=1}^n \sum_{\{w,w'\} \in \binom{S_i}{2}} \big(W_w(w')+W_{w'}(w)\big),\]
as desired.
\end{proof}

\begin{proof}[Proof of Lemma~\ref{lem:biasGen}]
From \eqref{eqn:wt} and \eqref{eqn:lip}, it follows that
\be
W_v(w) \le \frac{K_v}{d(v)}
\ee
with $W_v(w) = 0$ if $\{v,w\} \notin E(G)$. The lemma therefore follows from Lemma~\ref{lem:biasWeight}.
\end{proof}

\begin{proof}[Proof of Proposition~\ref{prop:biasExpected}]
By the linearity of expectation, we have
\begin{align*}
&\mathbb{E}_{\mathcal{P}} \Big[\sum_{v \in V(G)} \frac{|\mathcal{P}_v \cap \mathcal{N}(v)|}{d(v)} K_v\Big]= \sum_{v \in V(G)} \frac{K_v}{d(v)} \sum_{v \in e \in E(G)} \P(e \subseteq S_i \text{ for some } i)\\
=& \sum_{v \in V(G)} \frac{K_v}{d(v)} \sum_{v \in e \in E(G)} \frac{r-1}{rn-1} = \sum_{v \in V} \frac{K_v}{d(v)} \cdot d(v) \cdot \frac{(r-1)}{rn-1}\\
=& \frac{r-1}{rn-1} \sum_{v \in V(G)} K_v.
\end{align*}
The proposition follows, by Lemma~\ref{lem:biasGen}.
\end{proof}

\begin{proof}[Proof of Proposition~\ref{prop:biasExpectedTypes}]
By the linearity of expectation, we have
\begin{align*}
&\mathbb{E}_{\mathcal{P}} \Big[\sum_{v \in V(G)} \frac{|\mathcal{P}_v \cap \mathcal{N}(v)|}{d(v)} K_v\Big]= \sum_{v \in V(G)} \frac{K_v}{d(v)} \sum_{v \in e \in E(G)} \P(e \subseteq S_i \text{ for some } i)\\
\le& \sum_{\pi \in \Pi} \sum_{v \in \pi} \frac{K_v}{d(v)} \sum_{v \in e \in E(g)} \frac{r-1}{|\pi|-1} = \sum_{v \in V} \frac{K_v}{d(v)} \cdot d(v) \cdot \frac{(r-1)}{|\pi|-1}\\
\le& (r-1) K_{\max} \sum_{\pi \in \Pi} \frac{|\pi|}{|\pi|-1} \le 2 (r-1) K_{\max}
\end{align*}
The proposition follows, by Lemma~\ref{lem:biasGen}.
\end{proof}

\section{Bounds on MSE: dense case}

The following $L^2$ bound is the key to the proofs of all of the MSE bounds.

\begin{lemma}
\label{lem:L2pairs}
For all $\mathcal{P} = (S_1,\ldots,S_n) \in \binom{V(G)}{r,\ldots,r}$ and all $v,v' \in S_j,$ we have
\[
\begin{split}
\sqrt{\E_{\vec{B}}\left[\d\left(\vec{d}(v),\vec{d}(v')\right)^2 \mid v \in \T \text{ and } v' \notin \T\right]} &\leq
\frac{K_1}{\dmax}|d(v) - d(v')| \\
&\hspace{-5cm}+K_2\Big(\frac{2\sqrt{pq}}{r \sqrt{d(v)}} + \frac{2\sqrt{pq}}{r \sqrt{d(v')}} + \frac{q \cdot \ind_{E(G)}(\{v,v'\})}{r \cdot d(v)} + \frac{p \cdot \ind_{E(G)}(\{v,v'\})}{r \cdot d(v')}\Big)
\end{split}
\]
when $\T = \T_{\vec{B},\mathcal{P}}.$
\end{lemma}
\begin{proof}
Note that
\[\d\left(\vec{d}(v),\vec{d}(v')\right) = \frac{K_1}{\dmax}|d(v) - d(v')|+ \frac{K_2}{r}|F|,\]
where
\[F = \frac{r|\T \cap \mathcal{N}(v)|}{d(v)} - \frac{r|\T \cap \mathcal{N}(v')|}{d(v')}.\]
Thus, it suffices to prove that
\[\sqrt{\mathbb{E}_{\vec{B}}\left[F^2\mid v \in \T \text{ and } v' \notin \T\right]} \le \frac{2}{\sqrt{d(v)}} + \frac{2}{\sqrt{d(v') }}.\]

For $w \in V(G),$ let
\[F_w = \pq{T}{w} \left(\frac{\ind_{\mathcal{N}(v)}(w)}{d(v)} - \frac{\ind_{\mathcal{N}(v')}(w)}{d(v')}\right).\]
It is clear that
\begin{align*}
\mathbb{E}_{\vec{B}}\left[F_w ^2\mid v \in \T \text{ and } v' \notin \T\right] &= pq\left(\frac{\ind_{\mathcal{N}(v)}(w)}{d(v)} - \frac{\ind_{\mathcal{N}(v')}(w)}{d(v')}\right)^2\\
&\le 2pq\frac{\ind_{\mathcal{N}(v)}(w)}{d(v)^2} + 2pq\frac{\ind_{\mathcal{N}(v')}(w)}{d(v')^2}
\end{align*}
For all $i,$ let
\[F_i = \frac{r|\T \cap S_i \cap \mathcal{N}(v)| - p|S_i \cap \mathcal{N}(v)|}{d(v)} - \frac{r|\T \cap S_i \cap \mathcal{N}(v')| - p|S_i \cap \mathcal{N}(v')|}{d(v')}.\]

Note that $F_i = \sum_{w \in S_i} F_w$ and that
\[\E_{\vec{B}}\left[F_w\mid v \in \T \text{ and } v' \notin \T\right] = 0\]
for $w \not= v,v'$.
When $i \not= j$, we also have \[\Corr\left(\pq{\T}{w},\pq{\T}{w'} \mid v \in \T \text{ and } v' \notin \T\right) = -\frac{1}{r-1}\] for $w,w' \in S_i$.
Thus, for $i \not= j,$ we have
\begin{align*}
\E_{\vec{B}}\left[F_i^2 \mid v \in \T \text{ and } v' \notin \T\right] &\le 2 \sum_{w \in S_i} \mathbb{E}_{\vec{B}}\left[F_w^2\mid v \in \T \text{ and } v' \notin \T\right]\\
&\le \frac{4pq|S_i \cap \mathcal{N}(v)|}{d(v)^2} + \frac{4pq|S_i \cap \mathcal{N}(v')|}{d(v')^2}.
\end{align*}
Similarly, we have
\[\Corr\left(\pq{\T}{w},\pq{\T}{w'} \mid v \in \T \text{ and } v' \notin \T\right) = -\frac{1}{r-3}\]
for $w,w' \in S_j \setminus \{v,v'\},$ so that
\be
\E_{\vec{B}}&\left[(F_j-F_v-F_{v'})^2 \mid v \in \T \text{ and } v' \notin \T\right]\\
&\le 2 \sum_{w \in S_j \setminus \{v,v'\}} \mathbb{E}_{\vec{B}}\left[F_w^2\mid v \in \T \text{ and } v' \notin \T\right]\\
&\le \frac{4pq|S_j \cap \mathcal{N}(v) \setminus \{v'\}|}{d(v)^2} + \frac{4pq|S_j \cap \mathcal{N}(v') \setminus \{v'\}|}{d(v')^2}.
\ee
Since $B_1,\ldots,B_n$ are independent (even conditioned on $v \in \T$ and $v' \notin \T$), it follows that
\[\E_{\vec{B}}\Big[\big(\sum_i F_i-F_v-F_{v'}\big)^2 \mid v \in \T \text{ and } v' \notin \T\Big] \le \frac{4pq}{d(v)} + \frac{4pq}{d(v')}\]
so that
\[\sqrt{\E_{\vec{B}}\Big[\big(\sum_i F_i-F_v-F_{v'}\big)^2 \mid v \in \T \text{ and } v' \notin \T\Big]} \le \frac{2\sqrt{pq}}{\sqrt{d(v)}} + \frac{2\sqrt{pq}}{\sqrt{d(v')}}.\]

Noting that $F = \sum_i F_i$ and using the fact that  $|F_v| \le \frac{q \cdot \ind_{E(G)}(\{v,v'\})}{d(v)}$ and $|F_{v'}| \le \frac{p \cdot \ind_{E(G)}(\{v,v'\})}{d(v')},$ it follows that
\be
\sqrt{\E_{\vec{B}}\left[F^2 \mid v \in \T \text{ and } v' \notin \T\right]}&\\
&\hspace{-5cm}\le \frac{2\sqrt{pq}}{\sqrt{d(v)}} + \frac{2\sqrt{pq}}{\sqrt{d(v')}} + \frac{q \cdot \ind_{E(G)}(\{v,v'\})}{d(v)} + \frac{p \cdot \ind_{E(G)}(\{v,v'\})}{d(v')}
\ee
and the proof is finished.
\end{proof}

For $v \in S_i,$ define
\[D^v_\T = \ind_\T(v) \Big(q \delta_{\vec{d}(v)} - \sum_{w \in S_i \setminus \T} \delta_{\vec{d}(w)}\Big).\]
Note that
\[D^v_\T = \sum_{w \in S_i \setminus \{v\}} \ind_\T(v) \ind_{V(G) \setminus \T}(w) \left(\delta_{\vec{d}(v)} - \delta_{\vec{d}(w)}\right).\]

\begin{proof}[Proof of Proposition~\ref{prop:TVbound}]
We have
\be
\sqrt{\mathbb{E}_{\vec{B}}\left[\|D^v_\T\|_{\d}^2\right]}
\le&\sqrt{\frac{pq}{r(r-1)}} \sum_{w \in S_i \setminus \{v\}} \sqrt{\mathbb{E}\left[\d\left(\vec{d}(v),\vec{d}(w)\right)^2 \mid v \in \T \text{ and } v' \notin \T\right]}\\
\le&\frac{\sqrt{pq}}{r-1} \sum_{w \in S_i \setminus \{v\}} \sqrt{\mathbb{E}\left[\d\left(\vec{d}(v),\vec{d}(w)\right)^2 \mid v \in \T \text{ and } v' \notin \T\right]}.
\ee

By Lemma~\ref{lem:L2pairs}, it follows that
\be
\sqrt{\mathbb{E}_{\vec{B}}\left[\|D^v_\T\|_{\d}^2\right]} &\le
\frac{K_1 pq}{(r-1)\dmax} \sum_{w \in S_i \setminus\{v\}} |d(v)-d(w)|
+ \frac{2K_2 pq}{r \sqrt{d(v)}}\\
&\hspace{1cm} 
+ \sum_{w \in S_i \setminus \{v\}} \frac{2K_2 pq}{r (r-1)\sqrt{d(w)}} + K_2\sum_{w \in S_i \cap \mathcal{N}(v)} \left(\frac{q}{r \cdot d(v)} + \frac{p}{r \cdot d(w)}\right).\ee
Noting that $D_\T = \frac{1}{pqn} \sum_{v \in V(G)} D^v_\T,$ we have
\begin{align*}
\sqrt{\E_{\vec{B}}\left[\|D_\T\|_{\d}^2\right]} &\le \sum_{v \in V(G)} \sqrt{\E_{\vec{B}}\left[\|D^v_\T\|_{\d}^2\right]}\\
&\le \frac{K_1}{\sqrt{pq}n} \cp + \frac{1}{rn}\sum_{v \in V(G)} \frac{4K_2}{\sqrt{d(v)}} + \frac{K_2}{pqn} \sum_{v \in V(G)} \frac{|\mathcal{P}_v \cap \mathcal{N}(v)|}{d(v)},
\end{align*}
as claimed.
\end{proof}

\begin{proof}[Proof of Proposition~\ref{prop:TVboundTypes}]
As in the proof of Proposition~\ref{prop:TVbound}, we have
\be
\sqrt{\mathbb{E}_{\vec{B}}\left[\|D^v_\T\|_{\d}^2\right]} &\le
\frac{\sqrt{pq}}{r-1} \sum_{w \in S_i \setminus \{v\}} \sqrt{\mathbb{E}\Big[\dkp\left(\vec{d}(v),\vec{d}(w)\right)^2 \mid v \in \T \text{ and } v' \notin \T\Big]}.
\ee
By Lemma~\ref{lem:L2pairs}, it follows that
\be
\sqrt{\mathbb{E}_{\vec{B}}\left[\|D^v_\T\|_{\d}^2\right]} &\le
\frac{2K_2 pq}{r \sqrt{d(v)}}+ \sum_{w \in S_i \setminus \{v\}} \frac{2K pq}{r (r-1)\sqrt{d(w)}} \\
&\hspace{2cm}+ K\sum_{w \in S_i \cap \mathcal{N}(v)} \left(\frac{q}{r \cdot d(v)} + \frac{p}{r \cdot d(w)}\right).
\ee
Noting that $D^\pi_\T = \frac{1}{pqn} \sum_{v \in \pi} D^v_\T,$ we have
\begin{equation}
\label{eq:boundOnDpi}
\sqrt{\E_{\vec{B}}\left[\|D^\pi_\T\|_{\wdkp}\right]} \le \frac{1}{rn} \sum_{v \in \pi} \frac{4K}{\sqrt{d(v)}} + \frac{K}{pqn} \sum_{v \in \pi} \frac{|\mathcal{P}_v \cap \mathcal{N}(v)|}{d(v)}.
\end{equation}
Summing over $\pi \in \Pi,$ it follows that
\begin{align*}
\sqrt{\E_{\vec{B}}\Big[\big(\sum_{\pi \in \Pi} \|D^\pi_\T\|_{\d}\big]\Big)^2} &\le \sum_{\pi \in \Pi} \sqrt{\E_{\vec{B}}\left[\|D^\pi_\T\|_{\d}^2\right]}\\
&\le \frac{1}{rn}\sum_{v \in V(G)} \frac{4K}{\sqrt{d(v)}} + \frac{K}{pqn} \sum_{v \in V(G)} \frac{|\mathcal{P}_v \cap \mathcal{N}(v)|}{d(v)},
\end{align*}
as claimed.
\end{proof}

\section{Bounds on MSE: sparse case}
As Proposition~\ref{prop:biasExpected} shows, introducing randomness can help reduce bias. We will first need a generalization of Proposition~\ref{prop:biasExpected} to a class of semi-restricted randomizations.

Given a partition $\Pi$ of $V(G)$, let $\binom{\Pi}{r,\ldots,r}$ denote the set of partitions $\mathfrak{P} =(S_1,\ldots,S_k)$ of $V(G)$ into pairs such that $S_i$ lies in an element of $\Pi$ for every $i$.
That is, $\binom{\Pi}{r,\ldots,r}$ is the set of partitions of $V(G)$ into sets of size $r$ that refine $\Pi$.
Assume that the function $f_v$ is $K_v$-Lipshitz and define the quantity
\be \label{eqn:kmax}
K_{\max} = \max_{v \in V(G)} K_v.
\ee
\begin{prop}
\label{prop:biasExpectedTypes}
Fix a partition $\Pi$ of $V(G)$ into sets of size divisible by $r$.
When $\mathcal{P}$ is sampled uniformly from $\binom{\Pi}{r,\ldots,r}$, we have
\be
\mathbb{E}_{\mathcal{P}}\left|\mathbb{E}_{\vec{B}}(\xi | \mathcal{P})\right| \le \frac{2K_{\max} |\Pi|}{rn}
\ee
when $\T = \T_{\vec{B},\mathcal{P}}.$
\end{prop}
\subsection{With types}

The following lemma will be used in the proof of Proposition~\ref{prop:sparseMSEtypes}.
Recall $K_{\max}$ and $\dmin$ from Equations \eqref{eqn:kmax} and \eqref{eqn:dmin} respectively.
\begin{lemma}
\label{lem:VarOfEsemiRes}
When $\mathcal{P}$ is sampled uniformly from $\binom{\Pi}{r,\ldots,r}$, we have
\[\mathbb{E}_\mathcal{P}\big(\mathbb{E}_{\vec{B}}\left(\xi \mid \mathcal{P}\right)\big)^2 \le \frac{2K^2_{\max} |\Pi|}{nr\cdot\dmin}\]
when $\T = \T_{\vec{B},\mathcal{P}}$.
\end{lemma}
\begin{proof}
The proof of Proposition~\ref{prop:biasExpectedTypes} shows that
\be
\mathbb{E}_{\mathcal{P}}\Big[\sum_{v \in V(G)} \frac{K_v \cdot |\mathcal{P}_v \cap \mathcal{N}(v)|}{d(v)}\Big] \le 2 (r-1) K_{\max}|\Pi|.
\ee
We have
\be
\mathbb{E}_{\mathcal{P}}&\Big[ \Big\{\sum_{\{w_i,w'_i\} \in E(G)} \Big(\frac{K_{w_i}}{d(w_i)} + \frac{K_{w'_i}}{d(w'_i)}\Big)\Big\}^2\Big]\\
&\le \mathbb{E}_{\mathcal{P}}\Big[\sum_{v \in V(G)} \frac{K_v \cdot |\mathcal{P}_v \cap \mathcal{N}(v)|}{d(v)}\Big] \cdot \max_{\mathcal{P} \in \binom{\Pi}{r,\ldots,r}} \Big[\sum_{v \in V(G)} \frac{K_v \cdot |\mathcal{P}_v \cap \mathcal{N}(v)|}{d(v)}\Big]\\
&\le 2 (r-1) K_{\max} |\Pi| \cdot \frac{nr(r-1) K_{\max}}{\dmin}\\
&= \frac{2nr(r-1)^2K_{\max}^2|\Pi|}{\dmin}.
\ee
The lemma follows, by Lemma~\ref{lem:biasGen}.
\end{proof}

\begin{proof}[Proof of Proposition~\ref{prop:sparseMSEtypes}]
Let $\mathcal{P} = (S_1,\ldots,S_n)$ be sampled uniformly from $\binom{\Pi}{r,\ldots,r}$ and let $\T = \T_{\vec{B},\mathcal{P}}$.
For $i=1,2,\ldots,n,$ define
\[\xi_i = \sum_{v \in S_i} \xi_v.\]
Note that if $\xi_i$ and $\xi_j$ are dependent given $\mathcal{P}$ and $i \not= j$, then
either there is an edge between $S_i$ and $S_j$ or there exists $k$ such that there are edges between $S_i$ and $S_k$ and between $S_k$ and $S_j$.
In particular, for fixed $i,$ there are at most $r^2\dmax^2 + 1$ values of $j$ such that $\xi_i$ and $\xi_j$ are dependent given $\mathcal{P}$.
By Lemmata~\ref{lem:L2pairs} and~\ref{lem:varIndpt}, it follows that
\begin{align*}
\Var_{\vec{B}}\left(\xi \mid \mathcal{P}\right) &\le \frac{r^2\dmax^2+1}{p^2q^2n^2} \sum_{i=1}^n \Var_{\vec{B}}\left(\xi_i \mid \mathcal{P}\right)\\
&\le (r^2\dmax^2+1) \sum_{i=1}^n \E_{\vec{B}}\Big[\|D^{S_i}_\T\|_{\d}^2 \mid \mathcal{P}\Big].
\end{align*}
By (\ref{eq:boundOnDpi}) in the proof of Proposition~\ref{prop:TVboundTypes}
\begin{align*}
\Var_{\vec{B}}\left(\xi \mid \mathcal{P}\right) &\le (r^2\dmax^2+1) \sum_{i=1}^n \Big(\frac{1}{rn} \sum_{v \in S_i} \frac{4K}{\sqrt{d(v)}} + \frac{1}{pqn}\sum_{v \in S_i}\frac{|\mathcal{P}_v\cap\mathcal{N}(v)|}{d(v)}\Big)^2\\
&\le \frac{r^2\dmax^2+1}{n^2} \sum_{i=1}^n \left(\frac{4K}{\sqrt{\dmin}} + \frac{r \min\{r-1,\dmin\}}{pq \cdot \dmin}\right)^2\\
&= \frac{r^2\dmax^2+1}{n} \left(\frac{4K}{\sqrt{\dmin}} + \frac{r \min\{r-1,\dmin\}}{pq \cdot \dmin}\right)^2.
\end{align*}
Taking square-roots yields that
\be
|\sqrt{\Var_{\vec{B}}\left(\xi \mid \mathcal{P}\right)}| \le \frac{4K\sqrt{r^2\dmax^2+1}}{\sqrt{n \cdot \dmin}} + \frac{r\min\{r-1,\dmin\}\sqrt{r^2\dmax^2+1}}{pq\sqrt{n}\cdot \dmin},
\ee
as desired.

The bound on $\mathbb{E}_T\xi$ is given by Proposition~\ref{prop:biasExpectedTypes}.
It remains to prove the bound on $\E_\T\xi^2$.
Eve's Law, Lemma~\ref{lem:VarOfEsemiRes}, and the previous paragraph together imply that
\begin{align*}
\sqrt{\mathbb{E}_\T\xi^2} &= \sqrt{\mathbb{E}_{\mathcal{P}}\mathbb{E}_{\vec{B}}\left(\xi \mid \mathcal{P}\right)^2 + \mathbb{E}_{\mathcal{P}} \Var\left(\xi \mid \mathcal{P}\right)}\\
&\le \frac{K\sqrt{2|\Pi|}}{\sqrt{nr \cdot \dmin}} + \frac{4K\sqrt{r^2\dmax^2+1}}{\sqrt{n \cdot \dmin}} + \frac{r\min\{r-1,\dmin\}\sqrt{r^2\dmax^2+1}}{pq\sqrt{n}\cdot \dmin},
\end{align*}
as desired.
\end{proof}

\subsection{Without types}

The key to the proof of Proposition~\ref{prop:sparseMSE} is to note that the treatment group $T_{\vec{B},\mathcal{P}^{**}}$ has the same distribution as the treatment group $T_\Pi$ for a suitably chosen $\Pi$.

\begin{proof}[Proof of Proposition~\ref{prop:sparseMSE}]
Let $D = d(V(G) \setminus S)$.  For $d \in D,$ let $V_d = \{v \in V(G) \setminus S \mid d(v) = d\}.$
Define
\[\Pi = \left(S_1,\ldots,S_k, (V_d)_{d \in D}\right).\]

For $d \in D,$ define $g_{V_d} = \left.f\right|_{V_d}$.
For $1 \le i \le k$, define
\[g_{S_i}(a,b) = f\left(\left\lfloor \frac{a \cdot \max_{v \in S_i} d(v)}{a + b}\right\rfloor,\left\lceil \frac{b \cdot \max_{v \in S_i} d(v)}{a+b}\right\rceil\right).\]
It is straightforward to verify that $f$ and $g_{\Pi(v)}$ agree on $\{(a,b) \in \mathbb{Z}_{\ge 0}^2 \mid a + b = d(v)\}$
for all $v \notin S$
and
\[|f(a,b) - g_{\Pi(v)}(a,b)| \le K_1 \frac{|d(u)-\max_{v \in S_i} d(v)|}{\dmax} + \frac{K_2}{\dmin}\]
for all $a + b = d(u)$. Define
\[\zeta^\T_v = \pq{\T}{v} g_{\Pi(v)}\left(\vec{d}_T(v)\right)\]
and let
\[\zeta = \frac{1}{pqn} \sum_{v \in V(G)} \zeta^\T_v.\]
The discussion of the previous paragraph shows that
\be
|\zeta - \xi| \le \frac{K_1}{n} + \frac{K_2 (\dmax - \dmin)}{n\dmin}.
\ee
The proposition then follows by bounding $\zeta$ for the treatment $\T = \T_\Pi$ using Proposition~\ref{prop:sparseMSEtypes}.
\end{proof}

\section{Homophily}

\begin{proof}[Proof of Proposition~\ref{prop:knownTypes}]
The first assertion follows from Lemma~\ref{lem:idtUnbiased} because $\P(v \in \T) = \frac{p}{r}$ for all $v \in V(G)$.

As $|\T \cap \pi| = \frac{p|\pi|}{r}$ for all $\pi \in \Pi,$ we have
\be
\idt-\bar{t} = \frac{1}{pqn} \sum_{v \in V(G)} \left(\pq{\T}{v}\epsilon_v + \frac{pq}{r} \left(t_v-t_{\Pi(v)}\right)\right).
\ee
Note that $\Corr\left(\pq{\T}{v},\pq{\T}{w}\right) = -\frac{1}{r-1}$ if $v \not= w$ lie in a single part of $\Pi$ and $\Corr\left(\pq{\T}{v},\pq{\T}{w}\right) = 0$ if $v$ and $w$ lie in different parts of $\Pi$.
Thus, we have
\[\sum_{w \in V(G)} \left|\Corr\left(\pq{\T}{v},\pq{\T}{w}\right)\right| = 2\]
for all $v \in V(G)$.
By Lemma~\ref{lem:varCorr}, it follows that
\[\Var\left(\idt\right) \le \frac{1}{p^2q^2n^2}\sum_{v \in V(G)} 2 \Var\left(\pq{\T}{v}\right)\epsilon_v^2 = \frac{2r\sigma^2}{pqn},\]
as desired.
\end{proof}

%
%
%

\end{document}